\documentclass[11pt,leqno]{amsart}
\usepackage[english,activeacute]{babel}
\usepackage{epsfig}
\usepackage{graphicx}
\RequirePackage{amsthm,amsmath,amsfonts,amssymb}
\RequirePackage[numbers,sort&compress]{natbib}
\RequirePackage{nicefrac}
\usepackage{hyperref}
\usepackage[utf8]{inputenc}
\usepackage[T1]{fontenc}
\usepackage{amsmath,amsfonts,amssymb,mathrsfs,amsthm}
\usepackage{graphicx}
\usepackage[mathscr]{eucal}
\usepackage{graphics,graphicx}
\usepackage{textcomp}
\usepackage[usenames, dvipsnames]{color}
\usepackage[pagewise]{lineno}

\usepackage{bbold}

\def\e{{\varepsilon}}

\newcommand\R{\mathbb{R}}
\newcommand\C{\mathbb{C}}
\newcommand\N{\mathbb{N}}
\newcommand\Z{\mathbb{Z}}
\newcommand\T{\mathbb{T}}

\newcommand\D{\mathcal{D}}
\newcommand\X{\mathcal{X}}

\newcommand{\A}{\mathcal A}
\newcommand{\B}{B}
\newcommand{\id}{\operatorname{I}}

\newcommand\bna{\begin{eqnarray*}}
\newcommand\ena{\end{eqnarray*}}

\newcommand\bnan{\begin{eqnarray}}
\newcommand\enan{\end{eqnarray}}

\newcommand\bnp{\begin{proof}}
\newcommand\enp{\end{proof}}

\newcommand\bneq{\begin{eqnarray*}\left\lbrace \begin{array}{rll}}
\newcommand\eneq{\end{array} \right.\end{eqnarray*}}
\newcommand\bneqn{\begin{eqnarray}\left\lbrace \begin{array}{rll}}
\newcommand\eneqn{\end{array} \right.\end{eqnarray}}

\newcommand\bni{\begin{itemize}}
\newcommand\eni{\end{itemize}}

\newcommand\nor[2]{\left\|#1\right\|_{#2}}

\renewcommand\d{\partial}

\newcommand\M{\mathcal{M}}

\def\SS{(e^{tA})_{t\geq 0}}

\newtheorem{theorem}{Theorem}[section]
\newtheorem{remark}{Remark}[section]

\newtheorem{lemma}{Lemma}[section]

\newtheorem{prop}{Proposition}[section]
\newtheorem{assumptions}{Assumptions}[section]
\newtheorem{definition}{Definition}[section]

\newtheorem{corollary}[lemma]{Corollary}
\newtheorem{assumption}[lemma]{Assumption}
\theoremstyle{definition}

\newtheorem{assum}{Assumption} 
\newenvironment{assump}[2][]
  {\renewcommand\theassum{#2}\begin{assum}[#1]}
  {\end{assum}}
\begin{document}

\title{Periodic solutions for weakly damped systems}

\author[C. Laurent]{Camille Laurent}
\address{CNRS UMR 9008, Universit\'e Reims-Champagne-Ardennes, Laboratoire de Math\'ematiques de Reims (LMR), Moulin de la Housse-BP 1039, 51687 REIMS cedex 2, France}
\email{camille.laurent@univ-reims.fr}

 \author[I. Rivas]{Ivonne Rivas}
 \address{Universidad del Valle, Departamento de Matem\'aticas, Cali, Colombia}
 \email{ivonne.rivas@correounivalle.edu.co}

\begin{abstract}
In this article, we investigate the existence and properties of time-periodic solutions for damped evolutionary partial differential equations subject to periodic forcing. Particular emphasis is placed on configurations where the energy decay of the corresponding free equation is non-uniform. We link the speed of decay and the existence of periodic solutions for the forced equation.  Furthermore, we characterize the relationship between the resolvent growth and the associated loss of regularity. The theoretical framework is illustrated through several examples, including linear and nonlinear damped wave equations and coupled hyperbolic-parabolic systems. Finally, we provide a counterexample demonstrating the occurrence of resonance, in which regular but unbounded solutions emerge despite damping.
\end{abstract}

\maketitle
 \begin{center}
 {\it \thanks{Dedicated to Jean-Michel Coron on the occasion of his seventieth birthday, with profound admiration, gratitude, and friendship.}}
\end{center}
\vspace{0.3cm}

\vspace{0.5cm}


\tableofcontents
\section{Introduction}
In this article, we study time-periodic solutions of evolutionary partial differential equations subject to periodic forcing and damping. The existence of periodic solutions is a classical problem with a long history, and it remains highly relevant from an applied perspective.\\

Historically, an important direction of research has concerned conservative systems, such as nonlinear wave equations. In some situations, the free linear equation naturally admits a large family of periodic solutions, and it is natural to ask whether this persists for the associated nonlinear equations. There is an extensive literature on these questions for one-dimensional nonlinear wave equations. The pioneering works of Rabinowitz~\cite{R:78}, Brézis-Nirenberg~\cite{BN:78}, Brézis-Coron-Nirenberg~\cite{BCN:80}, and Coron~\cite{C:83} established the existence of nontrivial periodic solutions of $\partial_t^2 u-\partial_x^2u=g(u)$ in various settings, and have been highly influential in subsequent research. The literature is vast; we mention only, in higher dimensions, the recent results of Chatzikaleas-Smulevici~\cite{CS:24} and Berti-Langella-Silimbani~\cite{BLS:25}, and refer the reader to the references therein.\\
\medskip

Another class of results concerns free systems that exhibit natural damping under the influence of periodic perturbations. From a physical perspective, the central inquiry is whether a time-periodic input necessarily induces a corresponding time-periodic response within the system. Formally, we seek to establish the conditions under which resonance phenomena, defined herein as the emergence of unbounded solutions derived from bounded inputs, can be rigorously excluded. Furthermore, it is conjectured that, when a periodic solution exists, it may serve as a global attractor; specifically, trajectories are expected to converge to this unique periodic orbit, thereby necessitating its detailed investigation.\\

The literature on this subject is extensive. In the parabolic framework, we note the foundational contributions of Lunardi \cite{L:88}, Galdi and Sohr \cite{GS:04}, and Galdi and Kyed \cite{GK:18}. For hyperbolic equations of the damped-wave type, we refer the reader to Haraux \cite[Chapter 8]{H:Book18} and the citations therein. While high-intensity damping typically yields robust existence results for forced periodic solutions, numerous physical configurations involve dissipation mechanisms that are insufficient to ensure decay in the natural norm. A comprehensive introduction to these challenges is provided by Galdi et al. \cite{GMZZ:14}, where the authors investigate coupled hyperbolic-parabolic systems. In such models, dissipation is restricted to the parabolic component, a configuration the authors term \emph{partially dissipative systems}. Notwithstanding this specific focus, a significant portion of the analysis in \cite{GMZZ:14} addresses more general abstract systems of the form
\begin{equation}\label{PP}
    \begin{cases}
         \dfrac{d u}{dt}(t) = Au(t) + f(t), \quad t \in [0,+\infty),\\
         u(0)=u_0,
    \end{cases}
\end{equation}
considered in a Banach space $\mathcal{X}$, where $A$ denotes the infinitesimal generator of the semigroup $(e^{At})_{t \geq 0}$. The authors focus on the existence of periodic solutions $u: [0, \infty) \to \mathcal{X}$ for the abstract linear evolution equation subject to a periodic forcing term $f: [0, \infty) \to \mathcal{X}$. Under the typical assumption that the system is partially damped, the semigroup $(e^{At})_{t \geq 0}$ is assumed to be contractive (or dissipative), such that $\Vert u(t) \Vert_{\mathcal{X}} \leq \Vert u_0 \Vert_{\mathcal{X}}$ for all $t \geq 0$. Among the abstract frameworks developed, the following primary theorems are established:

\begin{itemize}
    \item if the system is \emph{uniformly stable}, see Definition~\ref{uniformlyS}, then, for any $T>0$ and $f\in L^{1}_{per}([0,T],\X)$, there exists a unique $T$-periodic solution of \eqref{PP}. (see \cite[Theorem 4.4]{GMZZ:14} or Bostan~\cite{B:09}).
    \item if the system is \emph{strongly stable}, see Definition~\ref{stronglyS}, then, for any $T>0$, there exists a dense set $\mathcal{Q}\subset L^{1}_{per}([0,T],\X)$, so that for any $f\in \mathcal{Q}$, there exists a unique $T$-periodic solution of \eqref{PP}, see \cite[Theorem 4.2]{GMZZ:14}. 
\end{itemize}
Numerous configurations exist in which systems exhibit strong stability without attaining uniform stability; coupled hyperbolic-parabolic systems serve as a particularly salient example and constitute the primary focus of their investigation. In such instances, a characterization of the subset $\mathcal{Q}$ or at least the establishment of sufficient conditions for membership therein is of significant theoretical importance. In \cite{GMZZ:14}, this set is characterized only in an abstract functional setting. Subsequently, Mosny, Muha, Schwarzacher, and Webster \cite{MMSW:24} addressed this problem for a specific class of coupled hyperbolic-parabolic systems (refer to Section \ref{s:exampleheatwave} for the precise model). Their findings demonstrate that, under certain geometric constraints, there exists a functional space of source terms that generate periodic solutions. Notably, this framework entails a loss of regularity; specifically, the source terms must be smoother than what is strictly required to generate a solution to \eqref{PP}.\\

The purpose of the present article is to make a bridge between the problem of finding periodic solutions and some non-uniform stability estimates.  Additionally, after establishing the relation with the decay rate, we connect it to some resolvent estimates.\\

\medskip

Let us now present our results more precisely.  Throughout the article, we will make the following assumptions without repeating them. \\
\begin{assumptions}\label{assumptions} Let $\X$ be a Banach space, and let $A: D(A)\subset \X \mapsto \X$ denotes a linear operator, which is the infinitesimal generator of the strongly continuous $C^0$ semigroup $(S(t))_{t\geq 0}=\SS$.  We assume  $\SS$ is a uniformly bounded semigroup, that is, there exists $M>0$, so that $\nor{S(t)}{\mathcal{L}(\X)}\leq M$ for all $t\geq 0$. Furthermore, we  assume the imaginary axis $i\R$ is contained in the resolvent set $\rho(A)$ (we refer to Section~\ref{s:not} for detailed notations).\\
\end{assumptions}
\begin{remark}
    Under Assumption~\ref{assumptions}, the semigroup $\SS$ is known to be strongly stable, see Definition~\ref{stronglyS}.  Moreover, as detailed in Lemma~\ref{t:decay}, these conditions ensure the existence of a non-uniform decay rate converging to zero as in \eqref{e:decayabtrsinterp} below.
\end{remark}
In light of the assumptions on the semigroup $(S(t))_{t \geq 0}$, we obtain the following result concerning the existence of periodic solutions:
\begin{theorem}\label{t:interp}
Let $\alpha>0$ and $\SS$ satisfying 
\begin{align}
\label{e:decayabtrsinterp}
\nor{e^{At}u_{0}}{\X}\leq h(t)\nor{u_{0}}{\D(A)}, \quad \forall t\geq 0, \quad \forall u_0\in \D(A),
\end{align}
for $h:[0,+\infty)\mapsto \R_+ $ continuous decreasing so that $h^{\alpha}\in L^{1}([0,+\infty))$. Then, for any $T>0$, $f\in L^{1}_{per}([0,T],D((-A)^{\alpha}))$, there exists an initial data $w_0\in \X$ such that the solution $u_{per}$ of \eqref{PP} with $u_{per}(0)=w_0$ is the unique $T$-periodic solution of \eqref{PP}. Moreover, there exists a constant $C>0$ such that 
\begin{equation*}
    \nor{w_0}{\X}\leq C\nor{f}{L^{1}_{per}([0,T],\D((-A)^{\alpha}))}.
\end{equation*}
The same result holds if  instead, $f\in W^{\lceil \alpha\rceil,1}_{per,0}([0,T],\X)$.\\
Moreover, 
for any solution $u$ to \eqref{PP}, we have
\begin{align*}
 \nor{u(t)-u_{per}(t)}{\X}   \underset{t\to+\infty}{\longrightarrow}0.
\end{align*}
\end{theorem}
We refer to \eqref{e:defper0} for the precise definition of the space $W^{k,1}_{per,0}([0,T],\X)$. Although this represents a simplification of the conditions, the result holds for $W^{k,1}_{per}([0,T],\X)$ under further weak assumptions.\\

In particular, for  the case of polynomial decreasing functions, we obtain the following Corollary. 
\begin{corollary}
\label{cor:decaypol}
    Under the assumptions of the previous Theorem~\ref{t:interp}, define 
    \begin{equation*}
       \beta_0:= \sup \left\{\beta\in \R_+; \ \exists C_{\beta}>0  \textnormal{  that \eqref{e:decayabtrsinterp} holds for }h(t)=\frac{C_{\beta}}{ t^{\beta}} \ \right\}\in [0,+\infty].  
    \end{equation*}
 If $\beta_0\neq 0$. Then, for either
\begin{itemize}
    \item $f\in L^1([0,T],\D((-A)^{\alpha}))$ for some $\alpha>1/\beta_0$,
    \item $f\in W^{k,1}_{per,0}([0,T],\X)$ for some $k\in \N$, $k>1/\beta_0$,
\end{itemize}
 the conclusion of Theorem~\ref{t:interp} holds.
\end{corollary}
It is well known by the works of Lebeau~\cite{L:96}, Batty-Duyckaerts~\cite{BD:08} and Borichev-Tomilov~\cite{BT:10} that the rate of decay of sufficiently smooth orbits for the corresponding semigroup $\SS$  can be associated with the size of the resolvent of $A$, $R(\lambda,
A)=(\lambda \id-A)^{-1}$,  with $\lambda$ on the imaginary axis.  This leads to the following result, establishing a connection between the resolvent bound and the loss of derivatives for the periodic source term.
\begin{theorem}
 \label{t:perresolv}
Assume that, for some $\alpha>0$, we have the resolvent bound $$\nor{R(i\eta ,
A)}{\mathcal{L}(\X)} \le C (1+|\eta|^{\alpha}), \quad \text{ for all} \quad \eta\in \R.$$
Then, for either 
\begin{itemize}
    \item $f\in L^1_{per}([0,T],\D((-A)^{\beta}))$, for some $\beta>\alpha$,
    \item $f\in W^{k,1}_{per,0}([0,T],\X)$, for some $k\in \N$, $k>\alpha$,
\end{itemize}
the conclusion of Theorem~\ref{t:interp} holds.
\end{theorem}
The theorems obtained here are, in fact, broader in scope, accommodating source terms in more versatile spaces, that we call admissible, and define in Section~\ref{s:admiss}. Motivated by applications to non-homogeneous periodic boundary forcing, we also consider source terms of the form $B g$, where $B$ is an admissible operator as commonly defined in control theory. We refer the reader to Theorem~\ref{t:variant} for a more comprehensive treatment of general source terms, and to Theorem~\ref{t:variantop} for results specifically suited to boundary control.\\

It is also possible to obtain some nonlinear results for small source terms as it is stated in the abstract result in Theorem~\ref{t:nonlin} in Section~\ref{s:nonlin}.   After all these abstract results, we give several applications in Section~\ref{s:example}:\\

\begin{itemize}
    \item \textbf{Damped Wave equation}: in Section~\ref{s:examplewave}, we give applications to the damped wave equation $$\partial_t^2 u-\Delta_g u+a(x)\partial_t u=0,$$ where various results are available, depending on the geometric setting. The decay is known to be enhanced when the trapped set (i.e., the set of geodesics avoiding the active damping zone) is meager and exhibits instability.  In our context, it means that when the trapped set is "small" and "unstable", the loss of regularity for the source term $f$ is smaller.  We refer to Theorem~\ref{t:dampedwave} for a more detailed statement.  We also prove some nonlinear results for a small forcing term in Theorem~\ref{t:wavenonlin}.
    \item \textbf{Damped Wave equation with non homogeneous boundary forcing}: in Section~\ref{s:examplewaveboundary}, we give some applications to the damped wave equation with some forcing term coming from a non homogeneous boundary value problem.  More precisely, we consider the damped wave equation $$\partial_t^2 u-\Delta_g u+a(x)\partial_t u=0,$$ subject to the nonhomogeneous Dirichlet boundary condition  
    $$u=f,\quad \textnormal{ on } \R_+\times \d\Omega,$$
    where $f$ is periodic.  We refer to Theorem~\ref{t:periodboundary} for the details.
   
    \item \textbf{Coupled parabolic/hyperbolic systems:} in Section~\ref{s:exampleheatwave}, we study some systems of coupled heat/wave equation 
    \bneq
\partial_t u-\Delta_g u=&0,\quad &\textnormal{ on }[0,T]\times \Omega_H\\
\partial_t^2 w-\Delta_g w=&0,\quad &\textnormal{ on }[0,T]\times \Omega_W,
\eneq
where the two solutions are coupled by some boundary conditions on an interface $\gamma$ where $\partial_t w=u$ and $\partial_{\nu_1} u=-\partial_{\nu_2}  w$.  In particular, the damping is only induced by the heat equation.  This decay shall be transferred to the wave equation by the boundary coupling which requires some geometric assumptions.  The result is stated in Theorem~\ref{t:perheatwave} after a more detailed introduction of the system in Section~\ref{s:exampleheatwave}
\end{itemize}
The study of the existence of a periodic solution $u$ is crucial to understand the asymptotic behavior of solutions of \eqref{PP}, since any solution $v$ to the same equation $\partial_{t}v=Av+f$ converge to this one, see Proposition \ref{p:convuniq} below.  But, we could wonder if the decay rate that we are asking, that is merely an integrable decay, is necessary. 
\medskip
We give one counterexample where there is a very weak decay, that is logarithmic.  The following theorem can be found in Laurent-L\'eautaud \cite{LL:21hypo} \footnote{The result is actually stated in a more general context of operators "sum of square" with analytic coefficients.  Yet, the abstract result allows to transfer this decay from resolvent estimates with observations that are well known without the analyticity for elliptic operators.  We refer to the comments in \cite[Theorem 1.5]{LL:21hypo}.}.
\begin{theorem}\cite[Theorem 1.5]{LL:21hypo}
~Let $\M$ be a compact manifold and $a\in L^{\infty}(\M,\R_+)$ is such that $a\geq  \delta  > 0$ a.e. on a
nonempty open set.  Then, any solution of 
\bneqn
\label{dampedSchrod}
i\partial_t u-\Delta_g u+ia(x)u&=&0,\\
u(0)&=&u_0\in L^2(\M),
\eneqn
satisfies $u(t)\underset{t\to +\infty}{\longrightarrow}0$ in $L^2(\M)$. If moreover for some $k\in \N^*$, $u_0\in \D(\mathcal{A}_S^k)$ with $\mathcal{A}_S=-i\Delta_g-a(x)$, then we have
\begin{equation*}
 \nor{u(t)}{L^2(\M)} \leq \frac{C_k}{\left[\log(t+2)\right]^{2k}} \nor{\mathcal{A}_S^k u_0}{L^2(\M)}.
\end{equation*}
\end{theorem}
Theorem~\ref{t:interp} cannot be applied in this case since $\left[\log(2+t)\right]^{-\alpha}$ is never integrable on $\R_+$. The next result shows that the conclusion of Theorem~\ref{t:interp} is indeed false in this situation.  There are a lot of resonant forcing terms, as regular as we want, so that the resulting solutions are always unbounded.  
\begin{theorem}
\label{t:counter}
Let $\M=\mathbb{S}^2$ and $a\in C^{\infty}(\M,\R_+)$ so that $a=0$ in a neighborhood of one equator.
Then, for $T=2\pi$ and any $k\in \N$, there is a $G_{\delta}$ dense subset $G$ of $L^1_{per}([0,T],H^k(\M))$ such that for any $f\in G$, the solution of 
\bneqn
\label{dampedSchrodper}
i\partial_t u-\Delta_g u+ia(x)u&=&f,\\
u(0)&=&0,
\eneqn
satisfies that the sequence $\left(\nor{u(nT)}{L^2}\right)_{n\in \N}$ is unbounded. In particular, for any $v_0\in L^2(\M)$ and $f\in G$, the solution of 
\bneqn
\label{dampedSchrodperu0}
i\partial_t v-\Delta_g v+ia(x)v&=&f,\\
v(0)&=&v_0,
\eneqn
is unbounded in $L^2(\M)$ and therefore not periodic. 
\end{theorem}
Our results demonstrate the existence of unbounded solutions for certain source terms in any Sobolev space.  However, this does not preclude the possibility that analytic periodic forcing terms might yield periodic solutions.  And the result would be consistent with the exponential growth of the resolvent, see \cite[Theorem 1.6]{LL:21hypo}.  Furthermore, we establish the existence of resonant solutions only for the fundamental period $T=2\pi$, which corresponds to the period of the undamped equation. It remains possible that alternative periods could lead to bounded solutions.  
\section{Admissibility}
\label{s:admiss}
In this section, we will define several concepts related to admissible source terms or admissible operators.  They will include the situation of Theorem~\ref{t:interp}, but might be more general, and it could be useful in other situations.
\subsection{Admissible source terms}
\begin{definition}\label{def:admiss}Let $E_2\subset \X$ be a Banach space and $T>0$. We say that a Banach  space $W \subset  L^1_{per}([0,T],\X)$ is admissible for $\SS$ and $E_2$, if there exists $C>0$ so that  for any $f\in W$, $\displaystyle\int_0^T e^{A(T-s)}f(s)ds\in E_2$ together with the estimate
    \begin{align}
\label{ineqfadmiX}
 \nor{ \int_0^T e^{A(T-s)}f(s)ds  }{E_2}\leq C \nor{f}{W}, \quad \forall f\in W.
\end{align}
\end{definition}
We will often write for simplicity "admissible for $E_2$" when the semigroup is implicit in the context. 
\begin{remark}
    The assumption $W \subset  L^1([0,T],\X)$ is only technical here. It allows us to show that the Cauchy problem with forcing $f\in W$ is well posed by standard semigroup theory.  It could be avoided in some cases where the equation $\partial_t u=Au+f$ is known to be well posed in $\X$ for $f\in W$.  This holds, for example, when the semigroup admits some regularizing effect.  For instance, for the heat equation on a bounded domain with Dirichlet boundary conditions, a source term in $L^2([0,T],H^{-1}(\Omega))$ allows us to construct weak solutions in $L^{\infty}([0,T],L^2(\Omega))$.  Our method would certainly apply in this context, but describing it in full generality is more technical.  We leave the details to the reader interested in such a context.
\end{remark}
Definition~\ref{def:admiss} is made useful by the following immediate property.
\begin{lemma}
\label{lm:L1E2admiss}
 Let $E_2\subset \X$ be a Banach space.   Assume that we have the inequality 
\begin{align}
\label{contE_2}
    \nor{e^{As}y}{E_2}\leq C\nor{y}{E_2}, \quad \forall s\in [0,T],\ \forall y\in E_2.
\end{align}
Then, $L^1_{per}([0,T],E_2)$ is admissible for $E_2$.
\end{lemma}
\bnp
The inclusion $W=L^1_{per}([0,T],E_2)\subset L^1_{per}([0,T],\X)$ is direct. Moreover, we have the estimate
\begin{align*}
\nor{\int_{0}^{T}e^{A(T-s)}f(s)ds}{E_2}&\leq \int_{0}^{T}\nor{e^{A(T-s)}f(s)}{E_2}ds\\
    &\leq C \int_{0}^{T}\nor{f(s)}{E_2}ds=C\nor{f}{ L^{1}([0,T],E_{2})}.
\end{align*} \enp
\begin{lemma}\label{lm:gainder}
Let $p\in [1,+\infty)$ and $E_3$ be a Banach space so that $L^p_{per}([0,T],E_3)$ is admissible for $\X$. For $k\in \N$  assume that $C^k_{per,0}([0,T],\D(A^k))\cap W^{k,p}_{per,0}([0,T],E_3)$ is dense $W^{k,p}_{per,0}([0,T],E_3)$.
If $f\in W^{k,p}_{per,0}([0,T],E_3)$ then $\int_0^T e^{A(T-s)}f(s)ds \in \D(A^k)$
 together with 
    \begin{align*}
 \nor{ \int_0^T e^{A(T-s)}f(s)ds  }{\D(A^k)}\leq C \nor{f}{W^{k,p}_{per}([0,T],E_3)}.
\end{align*}
In particular, $W^{k,p}_{per,0}([0,T],E_3)$ is admissible for $\D(A^k)$.
\end{lemma}
\bnp
Using integration by parts together with the periodic boundary conditions (see the comments in Section~\ref{s:not}, for the spaces defined in \eqref{e:defper0}), we can write, for $f\in C^k_{per,0}([0,T],\D(A^k))\cap L^p_{per}([0,T],E_3)$,
      \begin{eqnarray*}
    A^k \int_0^T e^{A(T-s)}f(s)ds &=& -\sum_{j=0}^{k-1} A^{k-1-j} \left( f^{(j)}(T) - e^{AT}f^{(j)}(0) \right) + \int_0^T e^{A(T-s)}f^{(k)}(s)ds \\
    &=&\int_0^T e^{A(T-s)}f^{(k)}(s)ds.
    \end{eqnarray*}
In particular, the admissibility condition \eqref{ineqfadmiX} provides
    \begin{align*}
 \nor{\int_0^T e^{A(T-s)}f(s)ds }{\X}\leq C\nor{\frac{\d^kf}{\d s ^k}}{L^p([0,T],E_3)}\leq  \nor{f}{W^{k,p}_{per}([0,T],E_3)}.
\end{align*}
This inequality gives the expected result for functions in a dense subset. We obtain the expected result by density.
\enp
 Following the same proof, we have actually proved the following more technical result.
\begin{lemma}\label{lm:gainder2}
Let $p\in [1,+\infty)$ and $E_3$ be a Banach space so that $L^p_{per}([0,T],E_3)$ is admissible for $\X$. For $k\in \N$  assume that $C^k_{per}([0,T],\D(A^k))$ is dense in 
\begin{equation*}W:=\left\{f\in W^{k,p}_{per}([0,T],E_3); f^{(j)}(0)\in \D(A^{k-j-1}), \forall j=0,...,k-1\right\}
\end{equation*} with its natural norm. If $f\in W$, then $\int_0^T e^{A(T-s)}f(s)ds \in \D(A^k)$ together with 
    \begin{align*}
 \nor{ \int_0^T e^{A(T-s)}f(s)ds  }{\D(A^k)}\leq C \nor{f}{W^{k,p}_{per}([0,T],E_3)}+\nor{\sum_{j=0}^{k-1} A^{k-j-1}\left[ f^{(j)}(0)-e^{AT}f^{(j)}(0)\right] }{\X}.
\end{align*}
\end{lemma}
\subsection{Admissible operators}
In this part, we are concerned with applications of boundary controllability, where the operators are typically unbounded. 
\begin{definition}\cite[Definition 4.2.1]{TucsnakWeiss}
Let $\X$ and $U$ be Hilbert spaces.  An operator $B\in \mathcal{L}(U,\X_{-1})$ is said to be an admissible control
operator for $\SS$ if for some $\tau>0$, the operator $\Phi_{\tau}\in \mathcal{L}(L^2([0,+\infty)),\X_{-1})$ defined by
\begin{equation*}
    \Phi_{\tau} =\int_0^{\tau}S(\tau-s) Bu(s)~ds,
\end{equation*}
satisfies $\operatorname{Ran} \Phi_{\tau}\subset \X$, where $\operatorname{Ran}$ denotes the range of $\Phi_{\tau}$.  
\end{definition}
In this case,  $\operatorname{Ran} \Phi_{t}\subset \X$ holds for any $t\geq 0$, see \cite[Proposition 4.2.2.]{TucsnakWeiss}. Moreover, for every $z_0 \in \X$ and $v \in L^2_{loc}([0,+\infty);U)$, the equation 
\begin{equation*}
    \begin{cases}
         \partial_t z(t) = Az(t) + B v(t), \quad t \in [0,+\infty),\\
         z(0)=z_0,
    \end{cases}
\end{equation*}
has a unique mild solution $z$ which belongs to $C([0,+\infty);\X)$, see \cite[Proposition 4.2.5.]{TucsnakWeiss}.
\begin{definition}\label{def:admiopforE2}Let $E_2\subset \X$ be a Banach space and $U$ be a Banach space. An operator $B\in \mathcal{L}(U,\X_{-1})$ is said to be an admissible control
operator for $\SS$ and $E_2$, if for  all $\tau>0$, the operator $\Phi_{\tau}\in \mathcal{L}(L^2([0,+\infty)),\X_{-1})$ defined by
\begin{equation*}
    \Phi_{\tau}u =\int_0^{\tau}S(\tau-s) Bu(s)~ds,
\end{equation*}
satisfies $\operatorname{Ran} \Phi_{\tau}\subset E_2$. We also assume that for all $\tau>0$, there exists $C_{\tau}$ so that 
   \begin{align*}
 \nor{ \Phi_{\tau}u }{E_2}\leq C_{\tau} \nor{u}{L^2([0,\tau],U)}, \quad \forall u\in L^2([0,\tau],U).
\end{align*}
\end{definition}
\begin{remark}
    The $L^2$  norm in time is adopted here due to its relevance in boundary control applications. However, an analogous definition can be established for the $L^p$ norm, where $p\in [1,+\infty]$.
\end{remark}

In the same spirit as Lemma~\ref{lm:gainder}, we obtain the following result.
\begin{lemma}
Let $U$ be a Banach space and $B\in \mathcal{L}(U,\X_{-1})$  be  an  an admissible control operator for $E_2$.  For $k\in \N$, assume that there exists a dense set $Y\subset U$, so that $BY\subset  \D(A^k)$ and $C^k_{per,0}([0,T],Y)$ dense in $ H^{k}_{per,0}([0,T],U)$. If $f\in H^{k}_{per,0}([0,T],U)$, then $\Phi_T(f)\in \D(A^k)$ and
    \begin{align*}
 \nor{ \Phi_T(f)}{\D(A^k)}\leq C \nor{f}{H^{k}_{per}([0,T],U)}.
\end{align*}
\bnp
The proof is similar to Lemma~\ref{lm:gainder} working on the dense set $C^k_{per,0}([0,T],Y)$ that is dense in $f\in H^{k}_{per,0}([0,T],U)$.
\enp
\end{lemma}
Note that similar computations were performed in \cite{EZ:10} concerning the regularity of the control.
\section{Uniqueness and convergence}
\begin{prop}
\label{p:convuniq}
 Let $f\in L^1_{per}([0,T],\X)$. Assume the existence of a $T$-periodic solution $u_{per}\in C^0_{per}([0,T],\X)$ of \eqref{PP}. Then, this solution is unique, and for any solution $u\in C^0(\R_+,\X)$ to \eqref{PP}, we have
\begin{align*}
 \nor{u(t)-u_{per}(t)}{\X}   \underset{t\to+\infty}{\longrightarrow}0.
\end{align*}
\end{prop}
\bnp
Let $u_2$ be another $T$-periodic solution of \eqref{PP}.  Then, $v=u_{per}-u_2$ is $T$-periodic and it can be written $v(t)=S(t)v(0)$. By periodicity, we get, 
\begin{align}
\label{e:uniqn}
    \nor{v(0)}{\X}=\nor{v(nT)}{\X}=\nor{S(nT)v(0)}{\X}.
\end{align}
We claim that \begin{gather}\label{claim}\nor{S(t)v_0}{\X}\underset{t\to+\infty}{\longrightarrow}0,\end{gather} for any $v_0\in \X$. This will end the proof of the first statement since it proves that $v(0)=0$, after letting $n$ go to $+\infty$ in \eqref{e:uniqn}. The second statement is obtained by considering $v=u-u_{per}$ (which is not periodic except if it is zero)  satisfies $v(t)=S(t)v(0)$.  We conclude using the claim \eqref{claim}.\\

We now review for the convenience of the reader the proof of \eqref{claim} (certainly already known).  By density of the domain, we can decompose $v_0=w+r$ with $w\in D(A)$ and $\nor{r}{\X}\leq \e$. We have
\begin{align*}
\nor{S(t)w}{\X}\leq \nor{S(t)A^{-1}}{\mathcal{L}(\X)}\nor{A w}{\X}\underset{t\to+\infty}{\longrightarrow}0,
\end{align*}
where we have used Theorem \ref{t:decay}. Moreover, denoting $C=\sup_{t\geq 0}\nor{S(t)}{\mathcal{L}(\X)}<+\infty$ which is finite by assumption, we can estimate
\begin{align*}
\nor{S(t)r}{\X}\leq \nor{S(t)}{\mathcal{L}(\X)}\nor{r}{\X}\leq C\e.
\end{align*}
This concludes the proof of the claim.
\enp
\section{Existence of periodic solutions for linear damped systems}
\subsection{From decay to the existence of periodic solutions}
\begin{theorem}\label{BasicIdea}
Let $E_{2}$ be a Banach space with $E_{2}\subset \X$. 
Assume that the semigroup $\SS$ satisfies
\begin{equation*}
\nor{e^{At}u_{0}}{\X}\leq h(t)\nor{u_{0}}{E_{2}}, \quad \forall t\geq 0,\ \forall u_0\in E_2,
\end{equation*}
for some $h\in L^{1}([0,+\infty))$ continuous decreasing. Let $T>0$. Assume also that we have the inequality 
\begin{align}
\label{contE_2thm}
    \nor{e^{As}y}{E_2}\leq C\nor{y}{E_2}, \quad \forall s\in [0,T], \ \forall y\in E_2.
\end{align}
Then, for any $f\in L^{1}_{per}([0,T],E_{2})$, there exists a unique initial data $w_0\in \X$  such that  the solution  of 
\bneqn
\label{PPthm}
\partial_{t}u(t)&=&Au(t)+f(t), \quad t \in [0,+\infty),\\
u(0)&=&w_0
\eneqn
is $T$-periodic. Moreover, there exists a constant $C>0$ so that 
\begin{equation*}
    \nor{w_0}{\X}\leq C\nor{f}{L^{1}([0,T],E_{2})}.
\end{equation*}
\end{theorem}
\begin{proof}
Initially, we need to find  $u_0\in \X$  so that the solution is periodic, that is, for all $n\in \N$, $u_{0}=u(nT)$.  By Duhamel formula, the solution of $\partial_t u=Au+f(t)$ is given by 
\bna
u(T)&=&e^{AT}u_{0}+\int_{0}^{T}e^{A(T-s)}f(s)ds=e^{AT}u_{0}+F_{T},
\ena
where we denote,  $$F_{T}=\int_{0}^{T}e^{A(T-s)}f(s)ds.$$
By the hypothesis and the assumption \eqref{contE_2thm}, $F_T\in E_2$ and satisfies
\begin{align*}
    \nor{F_T}{E_2}&= \nor{\int_{0}^{T}e^{A(T-s)}f(s)ds}{E_2}\leq \int_{0}^{T}\nor{e^{A(T-s)}f(s)}{E_2}ds\\
    &\leq C \int_{0}^{T}\nor{f(s)}{E_2}ds=C\nor{f}{ L^{1}([0,T],E_{2})}.
\end{align*}
By iteration, for $n\in\N$, we should have 
\begin{equation}
\label{e:unT}
\begin{split}
  u(nT)&=e^{nAT}u_0+\int_0^{nT}e^{A(nT-s)}f(s)ds\\
  &=e^{nAT}u_0+\int_0^{nT}e^{A(T-s +(n-1)T)}f(s)ds\\
&=e^{nAT}u_0+\sum_{k=1}^{n}\int_{(k-1)T}^{kT}e^{A\left((T-s)+(n-1)T\right)}f(s)ds\\
&=e^{nAT}u_0+\sum_{k=1}^{n}\int_{0}^{T}e^{A(T-\Tilde{s}+(n-k)T)}f(\Tilde{s}+(k-1)T)d\Tilde{s}\\
&=e^{nAT}u_0+\sum_{k=1}^{n}e^{A(n-k)T}\int_{0}^{T}e^{A(T-\Tilde{s})}f(\Tilde{s})d\Tilde{s}\\
&=e^{nAT}u_0+\sum_{k=1}^{n}e^{A(n-k)T}F_T=e^{nAT}u_0+\sum_{m=0}^{n-1}e^{A m T}F_T,
\end{split}
\end{equation}
by the $T$-periodicity of $f$.
\medskip
As we have verified in the proof of Proposition~\ref{p:convuniq},  $\|e^{nAT}u_0\|_{\X}\to 0$ as $n\to \infty$. We now formally define the following choice of initial datum: 
$$w_{0}:=\sum_{n=0}^{+\infty}e^{nAT}F_{T}.$$
$w_0$ is well defined in $\X$ when $F_T\in E_2$. Indeed, it is a sum normally convergent in $\X$ and together with the estimate
\begin{equation*}
\nor{w_{0}}{\X}\leq \sum_{n=0}^{+\infty}\nor{e^{AnT}F_{T}}{\X}\leq \sum_{n=0}^{+\infty} h(nT)\nor{F_{T}}{E_{2}}\leq C\nor{F_{T}}{E_{2}}\leq  C\nor{f}{ L^{1}([0,T],E_{2})},
\end{equation*}
with $h\in L^{1}([0,+\infty))$ a decreasing function, so that we have $\sum_{n=0}^{+\infty} h(nT)<+\infty$. It remains to check, that the proposed formula for the initial datum $w_0$ produces a periodic solution as expected.  For $u(T):=e^{AT}w_{0}+\int_{0}^{T}e^{A(T-s)}f(s)ds$, we have
\bna
u(T)-w_{0}&=&e^{AT}w_{0}+F_{T}-w_{0}\\
&=&e^{AT}\sum_{n=0}^{+\infty}e^{nAT}F_{T}+F_{T}-\sum_{n=0}^{+\infty}e^{nAT}F_{T}\\
&=&\sum_{n=0}^{+\infty}e^{(n+1)AT}F_{T}+F_{T}-\sum_{n=0}^{+\infty}e^{nAT}F_{T}\\
&=&\sum_{n=1}^{+\infty}e^{nAT}F_{T}+F_{T}-\sum_{n=0}^{+\infty}e^{nAT}F_{T}=0.
\ena
The uniqueness comes from Proposition~\ref{p:convuniq} (or by using the reasoning in \eqref{e:unT} which imposes a unique formula).
\end{proof}
In the proof of Theorem~\ref{BasicIdea}, we only used the fact that $F_T\in E_2$. This can be generalized to any set $W$ that is admissible, see Definition~\ref{def:admiss}.  It leads to the following variant that will be useful.
\begin{theorem}\label{t:variant}
Let $E_{2}$ be a Banach space with $E_{2}\subset \X$. Assume that the semigroup $\SS$ satisfies
\begin{equation*}
\nor{e^{At}u_{0}}{\X}\leq h(t)\nor{u_{0}}{E_{2}}, \quad \forall t\geq 0, \ \forall u_0\in E_2,
\end{equation*}
for some $h\in L^{1}([0,+\infty))$ continuous decreasing. Let $T>0$ and $W$ a Banach space  admissible for $E_2$ according to Definition \ref{def:admiss}. Then, for any $f\in W$, there exists a unique initial data $w_0\in \X$ such that the solution of \eqref{PPthm} is $T$-periodic.  Moreover, there exists a constant $C>0$ so that
\begin{equation}
\label{e:boundu0adm}
    \nor{w_0}{\X}\leq C\nor{f}{W}.
\end{equation}\end{theorem}
\bnp
Following the same notation as in the proof of Theorem~\ref{BasicIdea}, Definition \ref{def:admiss} gives $F_T\in E_2$ together with
\begin{align*}
    \nor{F_T}{E_2}&= \nor{\int_{0}^{T}e^{A(T-s)}f(s)ds}{E_2}\leq C \nor{f}{W}.
\end{align*}
The proof ends up exactly the same as in Theorem~\ref{BasicIdea}.
\enp
A similar variant holds for an admissible control operator.
\begin{theorem}\label{t:variantop}
Let $E_{2}$ be a Banach space with $E_{2}\subset \X$. Assume that the semigroup $\SS$ satisfies
\begin{equation*}
\nor{e^{At}u_{0}}{\X}\leq h(t)\nor{u_{0}}{E_{2}}, \quad \forall t\geq 0,\ \forall u_0\in E_2,
\end{equation*}
for some $h\in L^{1}([0,+\infty))$ continuous decreasing.
Let $U$ be a Banach space and $B\in \mathcal{L}(U,\X_{-1})$ an admissible control
operator for $E_2$, according to Definition~\ref{def:admiopforE2}. Then, for any $T>0$ and $f\in L^2_{per}([0,T],U)$, there exists a unique initial data $w_0\in \X$  such that  the solution  of 
\bneq
\partial_{t}u&=&Au+Bf,\\
u(0)&=&w_0
\eneq
is $T$-periodic. Moreover, there exists a constant $C>0$ so that
\begin{equation*}
    \nor{w_0}{\X}\leq C\nor{f}{L^2_{per}([0,T],U)}.
\end{equation*}\end{theorem}
In many applications, the space $E_2$  will be related to the domain of $A$ or $(-A)^{\alpha}$, as discussed in the introduction. One can formulate interpolation arguments in this context, leading to Theorem~\ref{t:interp}, which we now prove.
\bnp[Proof of Theorem~\ref{t:interp}]
By the interpolation arguments of Lemma~\ref{claimalpha} below, we get for all $ \alpha >0$,
\begin{equation*}
\nor{e^{At}u_{0}}{\X}\leq C_{\alpha}h(t/c_{\alpha})^{\alpha}\nor{u_{0}}{\D((-A)^{\alpha})}, \quad \forall t\geq 0, \forall u_0\in \D((-A)^{\alpha}) .
\end{equation*}
Apply Theorem~\ref{t:variant} to any space that is admissible for $\D((-A)^{\alpha})$.  We have a bound $M>0$ so that  $ \nor{e^{sA}}{\mathcal{L}(\X)}\leq  M$ for $s\in [0,T]$. By commutation and the definition of the norm, we have for $s\in [0,T]$,
\begin{align*}
    \nor{e^{As}y}{\D((-A)^{\alpha})}&= \nor{e^{As}y}{\X}+\nor{ (-A)^{\alpha} e^{As}y}{\X}=\nor{e^{As}y}{\X}+\nor{e^{As} (-A)^{\alpha} y}{\X}\\
    &\leq M \nor{y}{\X}+M\nor{ (-A)^{\alpha} y}{\X} \leq M\nor{y}{\D((-A)^{\alpha})}.
\end{align*}
Hence, \eqref{contE_2} is satisfied and Lemma~\ref{lm:L1E2admiss} gives that $L^{1}([0,T],\D((-A)^{\alpha}))$ is admissible for $\D((-A)^{\alpha})$, which shows the first result. Moreover, $L^{1}([0,T],\X)$ is admissible for $\X$, so Lemma~\ref{lm:gainder} implies that $W^{\lceil \alpha\rceil,1}_{per,0}([0,T],\X)$ is admissible for $\D((-A)^{\lceil \alpha\rceil})$. Since $\lceil \alpha\rceil\geq \alpha$, it implies $\D((-A)^{\lceil \alpha\rceil}) \subset \D((-A)^{\alpha})$ (see Lemma~\ref{equivnorm}). In particular, $W^{\lceil \alpha\rceil,1}_{per,0}([0,T],\X)$ is admissible for $\D((-A)^{\alpha})$, which gives the second result.
\enp
We used the following interpolation result.  We refer to Section~\ref{s:powerinterp} for more details.
\begin{lemma}\label{claimalpha}
Assume $\SS$ satisfies \eqref{e:decayabtrsinterp} for $h \in C^0(\R_+, \R_+)$ continuous decreasing. For all $\forall \alpha >0$, there exists $C_{\alpha }>0 $ so  
\begin{equation*}
\nor{e^{A t}u_{0}}{\X}\leq C_{\alpha } \left[h\left(  \frac{t}{\lceil \alpha\rceil}\right)\right]^{\alpha}\nor{u_{0}}{\D((-A)^{\alpha})}, \quad \forall t\geq 0, \forall u_0\in \D((-A)^{\alpha}) .
\end{equation*}
\end{lemma}
\bnp
We first take $\alpha=m \in \N^*$, and apply Lemma~\ref{lm:iterpower} with $T=e^{t A}$ and $B=-A$. Combining with \eqref{e:decayabtrsinterp}, it implies that
    \begin{align*}
        \nor{e^{A mt}u_{0}}{\X}\leq D_{m} h(t)^m\nor{u_{0}}{\D((-A)^m)}, \quad \forall t\geq 0, \ \forall u_0\in \D((-A)^{m}).
    \end{align*}
After rescaling, 
   \begin{align}\label{e:decayabtracintergersbis} 
        \nor{e^{At}u_{0}}{\X}\leq D_m h(t/m)^m\nor{u_{0}}{\D((-A)^m)}, \quad \forall t\geq 0, \ \forall  u_0\in \D((-A)^{m}).
    \end{align}
  Since $m=\lceil m\rceil$, the result for $\alpha$ integer is proved.  For $\alpha>0$ not an integer, we will use interpolation.  Define $m=\lceil \alpha \rceil$. Applying Lemma~\ref{lm:inter} with $T=e^{t A}$, $B=-A$ and $\theta=\alpha/m\in (0,1)$, interpolating between the bound \eqref{e:decayabtracintergersbis} and the uniform bound, we get
    \begin{align*}
    \nor{e^{t A}}{\mathcal{L}(\D((-A)^{\theta m}),\X)}\leq D_{m,\theta} M^{1-\theta}D_m^{\theta} h(t/m)^{\theta m}.
    \end{align*} 
Since $\theta m=\alpha$, this is the expected result.
\enp
\bnp[Proof of Corollary~\ref{cor:decaypol}]
By definition of $\beta_0$, we can pick $\gamma\in (\frac{1}{\alpha},\beta_0)$ so  there exists $C_{\gamma}>0$  that \eqref{e:decayabtrsinterp} holds for $h(t)=\frac{C_{\gamma}}{ t^{\gamma}}$.  In particular, using $\nor{e^{At}}{\mathcal{L}(\X)}\leq M$ for small $t$, we have for another constant $C$
\begin{equation*}
\nor{e^{At}u_{0}}{\X}\leq \frac{C}{\left<t\right>^{\gamma}}\nor{u_{0}}{\D(A)}, \quad \forall t\geq 0,  \ \forall u_0\in \D(A).
\end{equation*}
Denoting $h(t)=\frac{C}{\left<t\right>^{\gamma}}$, notice that $h^{\alpha}\in L^{1}([0,+\infty))$ for $\alpha \gamma>1$.
The application of Theorem~\ref{t:interp} gives the result.
\enp
\begin{remark}
We can typically have $\beta_0=+\infty$ if the decay $h$ is  exponential type $h(t)=e^{-t^{s}}$, which means that we lose an $\e$ derivative with $\e$ arbitrary small. We use the convention $\frac{1}{\beta_0}=0$ if $\beta_0=+\infty$.
\end{remark}
\subsection{Link with resolvent estimates} 
\bnp[Proof of Theorem~\ref{t:perresolv}]
Pick $\gamma\in (\alpha,\beta)$. Denoting $M$ the function in \eqref{mf} in the Appendix below, the bound on the resolvent can be written $M(\eta)\le C (1+\eta^{\alpha})$. Due to Theorem \ref{t:BD} of Batty-Duyckaerts, we obtain
\begin{equation*}
\nor{e^{At}A^{-1}}{\mathcal{L}(\X)}\leq C\left(\frac{\log(t)}{t}\right)^{\frac{1}{\alpha}}\leq C_{\gamma} \frac{1}{t^{\frac{1}{\gamma}}}, \quad t\geq B.
\end{equation*}
In particular, using $\nor{e^{At}}{\mathcal{L}(\X)}\leq C$ for small $t$,  we have
\begin{equation*}
\nor{e^{At}u_{0}}{\X}\leq C \frac{1}{\left<t\right>^{\frac{1}{\gamma}}}\nor{u_{0}}{\D(A)}, \quad \forall t\geq 0.
\end{equation*}
Denoting $h(t)=\dfrac{1}{\left<t\right>^{\frac{1}{\gamma}}}$, we notice that $h^{\beta}\in L^{1}([0,+\infty))$ since $\frac{\beta}{\gamma}>1$.
The application of Theorem~\ref{t:interp} with $\beta$ replacing $\gamma$ gives the result.
\enp
\begin{remark}
    There might be a direct proof of the previous result using functional calculus, without using the semigroup's decay.  This might allow for obtaining finer results in some scale of spaces fitting more closely to the function $M(\eta)$.  For instance, if $M$ has a logarithmic growth, we would expect to have a loss in some logarithmic Sobolev type spaces instead of losing  $\e$ derivatives.  We also expect that the conclusion of Theorem~\ref{t:perresolv} holds with $\beta=\alpha$, maybe under more restrictive assumptions on the space $\X$ and the operator $A$.
\end{remark}
\begin{remark}
In some cases where $A$ is decomposed as $A=A_0+BB^*$ where $A_0$ is skew-adjoint and $B^*B$ is the damping, some resolvent estimates can also be obtained by using the following inequality 
\begin{equation*}
    \nor{x}{\X}\leq M(\lambda)\nor{(A_0 +\lambda \id)x}{\X}+m(\lambda)\nor{B x}{\X}.
\end{equation*}
 Such results appeared in various abstract contexts in Anantharaman-L{\'e}autaud~\cite{AnanthLeautANPDE}, Joly-Laurent~\cite[Appendix]{JL:20} and Chill-Paunonen-Seifert-Stahn-Tomilov~\cite{CPSST:23}. Therefore, it is possible to obtain the existence of periodic solutions as a consequence of this type of resolvent estimates with observability.
\end{remark}
\section{Nonlinear solutions}
\label{s:nonlin}
\begin{theorem}\label{t:nonlin}
Under the same assumptions as in Theorem~\ref{t:variant}, if $F:C([0,T],\X)\longrightarrow W$ (nonlinear) is locally Lipschitz with $F(0)=0$ and
\begin{align}\label{boundF}
    \nor{F(u)-F(v)}{W}\leq C\left(\nor{u}{C([0,T],\X)}+\nor{v}{C([0,T],\X)} \right)\nor{u-v}{C([0,T],\X)},
\end{align}
for $u,v\in C([0,T],\X)$ with  $\nor{u}{C([0,T],\X)}\leq 1$ and  $\nor{v}{C([0,T],\X)} \leq  1$.
Then, there exists $\e>0$ so that for any $f\in W$ with $\nor{f}{W}\leq \e$, there exists an initial data $w_0\in \X$  such that  the solution  of 
\bneqn
\label{e:perNL}
\partial_{t}u&=&Au+F(u)+f,\\
u(0)&=&w_0,
\eneqn
is $T$-periodic. \end{theorem}
\bnp
We denote $L$ the linear continuous operator defined from $W$ to $\X$ described in Theorem~\ref{t:variant}, that is $L(g)=w_0$ so that the solution $u$ of 
\bneq
\partial_{t}u(t)&=&Au(t)+g(t), \quad t \in [0,+\infty),\\
u(0)&=&w_0
\eneq
is periodic. We denote $u=T(g)$,  the formula $(T(g))(t)=e^{tA}w_0+\int_0^t e^{A(t-s)}g(s)ds$ allows to perform the estimate
\begin{equation}\label{e:boundT}
    \nor{T(g)}{C([0,T],\X)}\leq C \nor{w_0}{\X}+ C\nor{g}{L^1([0,T],\X)}\leq C\nor{g}{W},
\end{equation}
where we have used \eqref{e:boundu0adm} and the embedding of $W$ in $L^1([0,T],\X)$. We want to find a solution of $u= T(f+F(u))$.  We will prove that $u\mapsto G(u):=T(f+F(u))$ is a contraction on the closed ball $B=\overline{B}_{C_{per}([0,T],\X)}(0,\eta)$ of $C_{per}([0,T],\X)$ for $\eta\leq 1$  to be chosen small enough.
For $u\in B$, we have, using \eqref{e:boundT} and \eqref{boundF} (with $v=0$),
\begin{equation*}
   \nor{G(u)}{C([0,T],\X)}=    \nor{T(f+F(u))}{C([0,T],\X)}\leq C\nor{f+F(u)}{W}\leq C\nor{f}{W}+C \nor{u}{C([0,T],\X)}^2\leq C\e+C\eta^2.
\end{equation*}
In particular, for $\eta=2C\e$ and $\e$ small enough, $G$ maps $B$ into itself. Similar estimates give for $u, v\in B$,
\begin{equation*}
   \nor{G(u)-G(v)}{C([0,T],\X)}=    \nor{T(F(u)-F(v))}{C([0,T],\X)}\leq C\nor{F(u)-F(v)}{W}\leq C\eta \nor{u-v}{C([0,T],\X)}.
\end{equation*}
Now, for $\e$ and therefore $\eta$ small enough, $G$ is contracting in the complete space $B$ and has a fixed point.  Note that we have used the fact that if $u$ is periodic, $G(u)$ is also periodic by the definition of $T$.  Also, a fixed point of $T$ is a solution of  \eqref{e:perNL} as expected.
\enp
\section{Examples}
\label{s:example}
\subsection{Damped wave equations}
\label{s:examplewave}
In this subsection, we are considering the following internally damped wave equation on a compact connected Riemannian manifold with (or without) boundary $\M$,
\bneqn
\label{dampedwave}
\partial_t^2 u-\Delta_g u+a(x)\partial_t u&=&0\quad \textnormal{ on }\R_+\times \M\\
u&=&0,\quad \textnormal{ on } \R_+\times \d\M\\
(u,\partial_t u)(0)&=&(u_0,u_1)\quad \textnormal{ on } \M.
\eneqn
where $a\in C^0(\M,\R_+)$ is not zero. In the case $\d\M=\emptyset$, the constant function $u(t)=1$ is never damped, so we are led to define on $\X=H^1_0\times L^2(\M)$ the seminorm
$$
|(u_0,u_1)|_\X^2 = \|\nabla u_0 \|_{L^2(\M)}^2 + \| u_1 \|_{L^2(\M)}^2,
$$
and the possible decay
\begin{equation}
\label{decaydampedwaveh}
 \left|(u,\partial_t u)(t)\right|_{H^1\times L^2(\M)} \leq h(t)  \nor{(u_0,u_1)}{H^2\times H^1(\M)}, \quad \forall (u_0,u_1)\in (H^2\cap H^1_0)\times H^1_0(\M).
\end{equation}
In each case, as in Corollary~\ref{cor:decaypol}, we will define an exponent $\beta_0$ as the best polynomial decay
    \begin{equation}
    \label{e:defbeta0waves}
       \beta_0:= \sup \left\{\beta\in \R_+; \exists C_{\beta}>0  \textnormal{ so that \eqref{decaydampedwaveh} holds for }h(t)=\frac{C_{\beta}}{ t^{\beta}} \ \right\}.  
    \end{equation}
We will consider the following situations:
\begin{enumerate}
\item 
The torus with flat geometry: $\M=\T^2$: there are many studies in this geometry. Anantharaman-L\'eautaud~\cite{AnanthLeautANPDE} proved \eqref{decaydampedwaveh} with $h(t)=\frac{C}{\sqrt{t}}$ for any non zero $a\in C^0(\M,\R_+)$ and $h(t)=\frac{C_{\e}}{t^{1-\e}}$ with some $\e>0$ depending on additional regularity assumption on $a$. We also refer to Burq-Hitrik~\cite{BurqHitrikMRL} for earlier results and also Sun~\cite{S:23} and Datchev-Kleinhenz-Prouff~\cite{DKP:25} for more precise recent results with geometrical assumptions on the damping.  In that case, Anantharaman-L\'eautaud~\cite{AnanthLeautANPDE} provides $\beta_0\in [1/2,1]$ in general, and some of these results state $\beta_0\geq 1-\e$ with an $\e$ depending on the geometric assumptions on the damping $a$.  We refer to the above references for more details.
    \item \label{i:convex}A bounded domain with convex obstacles with flat geometry: $\M=\Omega\setminus (\mathcal{O}_1\cup \mathcal{O}_2)$ where $\Omega\subset \R^d$ is a smooth domain and $\mathcal{O}_1$ and $\mathcal{O}_2$ are smooth disjoints strictly convex obstacles so that the convex hull $\operatorname{convhull}(\mathcal{O}_1\cup \mathcal{O}_2)$ of the obstacles is contained in $\Omega$. We assume $a\in C^1(\M,\R_+)$ so that $a(x)\geq \gamma>0$ for $x$ in a neighborhood of $\d \Omega$. Notice that there exists one periodic geodesic that does not meet the support of the damping.  The bound with $h(t)=Ce^{-ct^{1/3}}$  was proved in Joly-Laurent~\cite{JL:20} using as a black box some results for Schr\"odinger from Burq~\cite{B:93} and Burq-Zworski \cite{BZ:04}. Note that more obstacles can also be allowed with more assumptions. In that case, we have $\beta_0=+\infty$.
\item 
The peanut of rotation: We consider a compact two-dimensional manifold without boundary.  We assume that 
the damping $a$ is effective, that is uniformly positive, everywhere 
except in the central part of the manifold.  This part is a manifold of negative curvature and invariant by rotation along the $y-$axis. More precisely, let us set this part to be equivalent to the cylinder endowed with the metric $g(y,\theta)={\rm d}y^2+\cosh^2(y){\rm d}\theta^2$.
This central part admits a unique (up to change of orientation) periodic geodesic which is unstable; any other geodesic meets the support of the damping. This example has been studied, for example, in Christianson-Schenck-Vasy-Wunsch~\cite{CSVW:14}, Schenck~\cite{S:11}, and the references therein. In that case, the decay was proved with $h(t)=Ce^{-ct^{1/2}}$ and we have $\beta_0=+\infty$.
\item 
The open book: We consider the torus $\M=\T^2$ with flat geometry. The damping $a$ is 
assumed to depend only on the first coordinate and to be of the type 
$a(x)=|x_1|^{\gamma }$ with $\gamma>0$ to be chosen small enough.  In this case, 
there is a unique (up to change of orientation) geodesic which does not meet 
the support of the damping.  This example (actually a more general geometry) has been studied in L\'eautaud-Lerner~\cite{LL:17} and proved the decay $h(t)=\frac{C}{t^{1+\frac{1}{\gamma}}}$, that is $\beta_0\geq 1+\frac{1}{\gamma}$.
\item 
$\M$ is a compact connected Riemannian surface without boundary, whose geodesic flow has the Anosov property (for instance with negative Gauss curvature) with $a\in C^{\infty}(\M,\R_+)$ not zero. Dyatlov-Jin-Nonnenmacher~\cite{DJN:22} (see Jin~\cite{J:20} for earlier result for hyperbolic surfaces) obtained \eqref{decaydampedwaveh} with $h(t)=e^{-t}$, that is $\beta_0=+\infty$. 
We also refer to other results in any dimension with pressure conditions by Schenck~\cite{S:10}, following 
ideas of Anantharaman~\cite{AnantharamanAnnals}.
\end{enumerate}
Note that several authors also studied boundary damping, which might enter in our abstract result, see Phung~\cite{P:07} or Wang~\cite{W:23} for the cylindrical domain. 
\begin{theorem}
\label{t:dampedwave}
In the previous examples, with $\beta_0$ defined in \eqref{e:defbeta0waves}, fix $k\in \N$, $k>1/\beta_0$ and $T>0$. Let $f\in W^{k,1}_{per,0}([0,T],L^2(\M)))$. In the case $\d\M=\emptyset$, we make the additional assumption $\int_{\M}f(t,x)dx_g=0$ for all $t\in [0,T]$. Then, there exists $(u_0,u_1)\in H^1_0\times L^2(\M)$ so that the unique solution $u\in C^1(\R_+,L^2(\M))\cap C^0(\R_+,H^1_0(\M))$ to 
\bneqn
\label{dampedwaveforced}
\partial_t^2 u-\Delta_g u+a(x)\partial_t u&=&f\quad \textnormal{ on }\R_+\times \M\\
u&=&0,\quad \textnormal{ on } \R_+\times \d\M\\
(u,\partial_t u)(0)&=&(u_0,u_1)\quad \textnormal{ on } \M,
\eneqn
is $T$-periodic.
\end{theorem}
\bnp
We first deal with the case $\d\M\neq \emptyset$, which is slightly simpler since $|\cdot|_{\X}$ is a norm on $\X$. For $\X=H^1_0\times L^2(\M)$, we denote $\A=\begin{pmatrix}
0 & Id\\
\Delta_g & -a(x)
\end{pmatrix}$ with domain $\D(\A)=\left(H^2\cap H^1_0(\M)\right) \times H^1_0(\M)$, the generator of the damped wave. Corollary~\ref{cor:decaypol} implies the existence of periodic solutions for a source term $F=(0,f)\in W^{k,1}_{per,0}([0,T],\X)$. Proving the announced result in the case $\d\M\neq \emptyset$.

Now, we treat the case $\d\M=\emptyset$ where we have to pay attention to the Kernel of $\A$.  By Lemma \ref{lm:projdamped} below, we have $\operatorname{Ker}(\A)=  \operatorname{Vect}(e_0,0)$ where $e_0$ is the constant function normalized in $L^2(\M)$. This leads us to introduce the spectral projector of $\A$ onto the spectral subspace of $\A$ associated to the eigenvalue $0$, namely 
$$
\Pi_0 = \frac{1}{2i\pi}\int_{\gamma} (z Id - \A)^{-1} dz \in \mathcal{L}(\X),
$$
where $\gamma$ denotes a positively oriented circle centered on $0$ with a radius so small that $\sigma(\A)\cap \gamma =\emptyset$ and $0$ is the single eigenvalue of $\A$ in the interior of $\gamma$. 

We set $\dot{\X} = (\id - \Pi_0)\X$ and equip this space with the norm
$$
\nor{(u_0 , u_1)}{\dot{\X}}^2 := |(u_0 , u_1)|_{\X}^2 
= \|\nabla u_0\|_{L^2(\M)}^2 + \|u_1 \|_{L^2(\M)}^2 ,
$$
and associated inner product. This is indeed a norm on $\dot{\X}$ since $\|(u_0 , u_1)\|_{\dot{\X}} =0$ is equivalent to $(u_0 , u_1)\in \operatorname{Ker}(\A)= \Pi_0\X$. Besides, we set $\dot{\A} = \A|_{\dot{\X}}$ with domain $D(\dot{\A}) = D(\A) \cap \dot{\X}$. 
Remark that $\sigma(\dot{\A}) = \sigma(\A) \setminus \{0\}$ and thus $\sigma(\dot{\A}) \cap i\R = \emptyset$, see \cite[Lemma 4.2.]{AnanthLeautANPDE} for a description of the spectrum of $\A$. 

Our forced equation~\eqref{dampedwaveforced} can be written in an abstract way as in \eqref{PP} with $F=(0,f)$.  The precise description of the projector $\Pi_0$ has been written by L\'eautaud in \cite{L:23} that we reproduce in Lemma~\ref{lm:projdamped} below. In particular, the zero-means assumption on $f(t)$ for every $t$ provides $\Pi_0 F(t)=0$ for all $t\in [0,T]$. This means that $F(t)\in \dot{\X}$ for any $t\in [0,T]$. The decomposition of the semigroup \eqref{e:decomposemigroup} in Lemma~\ref{lem:semigroup} below gives, denoting $U_0=(u_0,u_1)$, that for any $t\in [0,T]$, we have
\begin{align*}
U(t)=&(u(t),\partial_t u(t))= e^{\A t}U_{0}+\int_{0}^{t}e^{\A(t-s)}F(s)ds\\
=&e^{t\dot{\A}} (\id -\Pi_0)U_0 +\int_{0}^{t}e^{(t-s)\dot{\A}}(\id -\Pi_0) F(s)ds \\
&+\Pi_0 (u_0,u_1)+  \int_{0}^{t}\Pi_0 F(s)ds\\
=&e^{t\dot{\A}} (\id -\Pi_0)U_0 +\int_{0}^{t}e^{(t-s)\dot{\A}} F(s)ds +\Pi_0 (u_0,u_1).
\end{align*}
It means that the component along $\operatorname{Ker}(\A)$ is constant.  So finding a periodic solution is equivalent to finding a periodic solution for the operator $\dot{\A}$ on $\dot{\X}$.  But, we check that if $C>0$ is so that $\nor{a(x)u_1}{L^{2}(\M)}\leq C \nor{u_1}{L^2(\M)}$, we have for any $(u_0,u_1)\in \D(\A)$,
\begin{align*}
  \nor{(u_0,u_1)}{H^2\times H^1(\M)}\leq &\nor{(u_0,u_1)}{H^1\times L^2(\M)}+ \nor{\Delta_g u_0}{L^2}+\nor{u_1}{H^1}\\
 \leq  & \nor{(u_0,u_1)}{H^1\times L^2(\M)}+ \nor{\Delta_g u_0-a(x)u_1}{L^2}+ C\nor{u_1}{L^2}+\nor{u_1}{H^1}\\
 \leq & (C+2)\nor{(u_0,u_1)}{\D(\A)},
\end{align*}
where we have used the norm for $\D(\A)$, $\nor{(u_0,u_1)}{\D(\A)}=\nor{(u_0,u_1)}{H^1\times L^2(\M)}+\nor{u_1}{H^1_0}+\nor{\Delta u_0-a(x)u_1}{L^2}$. In particular, the decay \eqref{decaydampedwaveh} can be written
\begin{equation*}
\nor{e^{t\dot{\A}}U_0}{\dot{\X}} \leq C h(t)  \nor{U_0}{\D(\A)}, \quad \forall U_0=(u_0,u_1)\in D(\dot{\A}).
\end{equation*}
So, we are in conditions to apply Corollary~\ref{cor:decaypol} and conclude.
\enp
Note that in the case $\d\M=\emptyset$, we have lost the uniqueness of the periodic function, but only up to the one dimensional space $\operatorname{Vect}(e_0,0)$. See also Remark~\ref{rk:uniqKG}.  During the proof, we have used the following description of the semigroup $\A$.  We follow the notations introduced in the proof. 
\begin{lemma}[Lemma~4.3 of Anantharaman-L{\'e}autaud~\cite{AnanthLeautANPDE}]
\label{lem:semigroup}Assume $\d\M=\emptyset$. The operator $\dot{\A}$ generates a contraction $C^0$-semigroup on~$\dot{\X}$, denoted $(e^{t\dot{\A}})_{t\geq 0}$. Moreover, the operator $\A$ generates a bounded $C^0$-semigroup on~$\X$, denoted $(e^{t\A})_{t\geq 0}$ and the unique solution to~\eqref{dampedwave} is given by $(u, \d_t u )(t) = e^{t\A}(u_0,u_1)$. Finally,  we have
\begin{equation}
\label{e:decomposemigroup}
e^{t\A}= e^{t\dot{\A}} (\id -\Pi_0)  + \Pi_0  , \quad \text{for all } t\geq 0 .
\end{equation}
\end{lemma} 
\begin{lemma}[Lemma 2.9 of L{\'e}autaud~\cite{L:23}]
\label{lm:projdamped}Assume $\d\M=\emptyset$.
Assume $a\geq 0$ and $a$ does not vanish identically.
Then, we have $\operatorname{Ker}(\A)=  \operatorname{Vect}(e_0,0)$ where $e_0$ is the constant function normalized in $L^2(\M)$ and the spectral projector of $\A$ associated to the eigenvalue $0$ writes:
$$
 \Pi_{0}
\left(
\begin{array}{c}
u_0 \\ u_1
\end{array}
\right) 
=\left(\frac{1}{\int_{\M} a} \int_M (a(x)u_0(x) + u_1(x))dx_g\right)
\left(
\begin{array}{c}
1 \\ 0
\end{array}
\right).
$$
\end{lemma} 
\begin{remark}\label{rk:uniqKG}
In the case $\d\M=\emptyset$, it would be possible to avoid this problem of the undamped constants by considering the Klein-Gordon operator $\partial_t^2 u-\Delta_g u+mu+a(x)\partial_t u$ with $m>0$.  In that case, we could use \cite[Proposition B.4.]{JL:20}, which, in some cases, allows us to transfer resolvent estimates for the damped wave operator to a bounded perturbation.
\end{remark}
 Now, we deal with nonlinear equations.  There are several results on the existence of periodic solutions for damped nonlinear wave type equations, but, to our knowledge, they all concern cases with strong exponential decay, often with constant damping.  We refer to the book of Haraux~\cite[Chapter 8]{H:Book18} for a recent presentation of these results.  Now, we present some nonlinear results for a small forcing term.
\begin{theorem}
\label{t:wavenonlin}
Let $g$ be a polynomial function $\R\longrightarrow \R$ with $g(0)=g'(0)=0$. In the case \eqref{i:convex} of a domain with convex obstacles in dimension $2$.  For every $0<\sigma<1/2$, there exists  $\e>0$ so that  for every $f\in L^{1}_{per}([0,T],H^{\sigma}(\M))$ with $\nor{f}{L^{1}_{per}([0,T],H^{\sigma}(\M))} \leq \e$, there exists  $(u_0,u_1)\in H^1_0\times L^2(\M)$ so that the unique solution $u\in C^1(\R_+,L^2(\M))\cap C^0(\R_+,H^1_0(\M))$ to 
\bneq
\partial_t^2 u-\Delta_g u+a(x)\partial_t u&=&f+g(u)\quad \textnormal{ on }\R_+\times \M,\\
u&=&0\quad \textnormal{ on } \R_+\times \d\M,\\
(u,\partial_t u)(0)&=&(u_0,u_1)\quad \textnormal{ on } \M,
\eneq
is $T$-periodic.
\end{theorem}
\bnp
We will apply Theorem~\ref{t:nonlin} with the same $\X$ and $\A$ defined in the proof of Theorem~\ref{t:dampedwave}.  In particular, $\D((-\A)^{\sigma})=H^{1+\sigma}_0\times H^{\sigma}(\M)$ for $0<\sigma<1/2$. We pick $W=L^1([0,T],\D((-A)^{\sigma})$. For $U=(u,\dot{u})\in C([0,T],\X)$, we define $F(U)=(0,g(u))$ and prove that $F$ applies $C([0,T],\X)$ to $W$ and satisfies \eqref{boundF}. Using Lemma~\ref{lm-compo} below, we can bound 
\begin{align*}
   \nor{F(U)-F(V)}{W}&=\nor{g(u)-g(v)}{L^1([0,T],H^{\sigma}(\M))} \\
 &\leq C\left(\nor{u}{C([0,T],H^1(\M))}+\nor{v}{C([0,T],H^1(\M))} \right)\nor{u-v}{C([0,T],H^1(\M))}\\
&\leq C\left(\nor{U}{C([0,T],\X)}+\nor{V}{C([0,T],\X)} \right)\nor{U-V}{C([0,T],\X)}.
\end{align*}

\enp

We used the following result of composition, which is a slight variant of \cite[Proposition 3.2]{JL:20}.
\begin{lemma}\label{lm-compo}
Assume that dim$(\M)=2$. Then for any $\sigma\in [0,1)$, the function $g$ as above maps any bounded set of $H^1_0(\M)$ in a bounded of $H^\sigma_0(\M)$ and we have the estimates 
\begin{align*}
\nor{g(u)-g(v)}{H^{\sigma}_D(\M)}&\leq C\left(\nor{u}{H^1(\M)}+\nor{v}{H^1(\M)} \right)\nor{u-v}{H^1(\M)}.
\end{align*}
\end{lemma}
\bnp
We can assume without loss of generality that $\sigma\in (1/2,1)$, so that $H^{\sigma}_D(\M)=H^{\sigma}_0$. Therefore, since  $u=0$ on $\d\M$ implies  $g(u)=0$ on $\d\M$ it is enough to prove the bound of the norm $H^{\sigma}(\M)$. 
Since $f$ is polynomial function with $g(0)=g'(0)=0$, it is a finite sum of monomial $x^p$ with $p\in \N$, $p\geq 2$. So, it is enough to consider $g=u^p$ for such $p$. Moreover,  thanks to the Sobolev embedding $W^{1,q}(\Omega) \hookrightarrow H^{\sigma}(\M)$ for  $q=\frac{2}{2-\sigma}\in [1,2)$, we get
\begin{align*}
\nor{u^p-v^p}{H^{\sigma}_D(\M)}\leq&  \nor{u^p-v^p}{W^{1,q}(\M) }\leq  \nor{(u-v)\sum_{i=0}^{p-1} u^iv^{p-1-i}}{W^{1,q}(\M)}\\
 \lesssim &  \nor{\nabla (u-v)}{L^2}\left(\nor{u}{L^{r}}^{p-1}+\nor{u}{L^{r}}^{p-1}\right) \\
 &+   \nor{u-v}{L^r}\left(\nor{\nabla u}{L^{2}}+\nor{\nabla v}{L^{2}}\right)\left(\nor{u}{L^{r}}^{p-2}+\nor{u}{L^{r}(\M)}^{p-2}\right) \\
 &+\nor{u-v}{L^2}\left(\nor{u}{L^{r}}^{p-1}+\nor{u}{L^{r}(\M)}^{p-1}\right)\\
\lesssim  &\left(\nor{u}{H^1(\M)}+\nor{v}{H^1(\M)} \right)\nor{u-v}{H^1(\M)}
\end{align*}
with $r=(p-1)\frac {2q}{2-q}\in [1,+\infty)$ defined as soon as $1\leq q<2$ is chosen so that $\frac{1}{q}=\frac{1}{2}+\frac{p-1}{r}$. We have used the Sobolev embedding $H^{1}(\M) \hookrightarrow L^r(\M)$ in the last line. This finishing the proof of the Lemma.
\enp
\subsection{Damped waves with boundary source term}
\label{s:examplewaveboundary}
The first example we give is a case where the damping is strong, but there is a boundary forcing term.  It seems the result is new in this case.
\begin{assump}{GCC}\label{assumGCC}
\textit{We say that an open subset $\omega\subset \Omega$ satisfies the Geometric Control Condition if there exists $T_{\omega}>0$ so that every generalized geodesic of $\Omega$ traveling at speed 1 meets $\omega$ in time $t\in (0, T)$.}
\footnote{The ray of Geometric Optics are defined in Melrose-Sjostrand~\cite{MS:78}.  We will always assume that the Hamiltonian vector field of the wave operator does not have
contact of infinite order with $\d \M$.  Which ensures that the broken bicharacteristic flow is uniquely defined.}

 \textit{We say that an open subset $\gamma\subset \d\Omega$ satisfies the Geometric Control Condition if there exists $T_{\gamma}>0$ so that every generalized geodesic of $\Omega$ traveling at speed 1 meets $\gamma$ at least one time at a point not strictly diffractive in time $t\in (0, T)$.}
\end{assump}
\begin{theorem}
\label{t:periodboundary}
    Let $\Omega$ be a bounded domain of $\R^d$ with smooth boundary.  Let $a\in C^1(\M,\R_+)$ so that $a(x)\geq \eta>0$ on an open subset $\omega\subset \Omega$ satisfying the Geometric Control Condition \eqref{assumGCC}. Then for any $T>0$ and $f\in L^2_{per}([0,T],L^2(\d \Omega))$, there exists a unique $(u_0,u_1)\in L^2\times H^{-1}(\Omega)$ so that the unique solution $u\in C^1(\R_+,H^{-1}(\Omega))\cap C^0(\R_+,L^2(\Omega))$ to 
\bneqn
\label{dampedwaveboundary}
\partial_t^2 u-\Delta u+a(x)\partial_t u&=&0\quad \textnormal{ on }\R_+\times \Omega,\\
u&=&f\quad \textnormal{ on } \R_+\times \d\Omega,\\
(u,\partial_t u)(0)&=&(u_0,u_1)\quad \textnormal{ on } \Omega,
\eneqn
is $T$-periodic. 
\end{theorem}
\bnp
Note that the notion of solution here is by transposition.  We denote $\X=H^1_0\times L^2(\Omega)$, and $\A=\begin{pmatrix}
0 & \id\\
\Delta & -a
\end{pmatrix}$ with domain $\D(\A)=\left(H^2\cap H^1_0(\M)\right) \times H^1_0(\M)$, the generator of the damped wave equation with Dirichlet boundary conditions. Denoting $\Delta^{-1}$ the inverse of the Dirichlet Laplacian that is well defined from $L^2$ to $H^2\cap H_0^1$ as well as from $H^{-1}$ to $H^1_0$. We compute $\A^{-1}=\begin{pmatrix}
\Delta^{-1}\circ  a &\Delta^{-1}\\
\id &0  
\end{pmatrix}$ from   $\X$ to $\D(\A)$, that can be extended from $\X_{-1}$ to $\X$ (see \cite[Proposition 2.10.3.]{TucsnakWeiss}). 
\begin{align*}
    \nor{(f_0,f_1)}{\X_{-1}}&=\nor{A^{-1}(f_0,f_1)}{\X}=\nor{\Delta^{-1}(a(x)f_0+f_1)}{H_0^1}+\nor{f_0}{L^2(\Omega)}\\
    &=\nor{f_1+a(x)f_0}{H^{-1}(\Omega)}+\nor{f_0}{L^2(\Omega)}\leq \nor{f_1}{H^{-1}(\Omega)}+(C+1)\nor{f_0}{L^2(\Omega)}.
\end{align*}
where $C>0$ is one constant so that $\nor{a(x)f_0}{H^{-1}(\Omega)}\leq C \nor{f_0}{L^2(\Omega)}$. And reciprocally,
\begin{align*}
(1+C)\nor{(f_0,f_1)}{\X_{-1}}&\geq \nor{f_1+a(x)f_0}{H^{-1}(\Omega)}+(1+C)\nor{f_0}{L^2(\Omega)}\geq \nor{(f_0,f_1)}{L^2\times H^{-1}(\Omega)}.
\end{align*}
In particular, the norms of $\X_{-1}$ and $L^2\times H^{-1}(\Omega)$ are equivalent and $\X_{-1}=L^2\times H^{-1}(\Omega)$ since both spaces are the completion of  $\X$ for their respective norms.

 By the work of Bardos-Lebeau-Rauch~\cite{BLR}  (see also \cite{L:96}) the semigroup $\A$ satisfies $\nor{e^{t\A}x_0}{\X}\leq Ce^{-\lambda t}\nor{x_0}{\X}$ for some $C$, $\lambda>0$. In particular, $\nor{e^{t\A}\A^{-1}x_0}{\X}\leq Ce^{-\lambda t}\nor{\A^{-1}x_0}{\X}$, which is $\nor{e^{t\A}x_0}{\X^{-1}}\leq Ce^{-\lambda t}\nor{x_0}{\X^{-1}}$ for any $x_0\in \X$. This extends by density to $\X^{-1}$ after extending the semigroup to $\X^{-1}$ (see for instance \cite[Proposition 2.10.2.]{TucsnakWeiss}).

We want to apply Theorem~\ref{t:variantop} on $\X_{-1}=L^2\times H^{-1}(\Omega)$ and $U=L^2(\d\Omega)$. It only remains to describe how we can rewrite the boundary value problem as an admissible operator $B$.  This is done precisely in \cite[Chapter 10.9]{TucsnakWeiss} for the free wave operator.  Similar arguments work for the damped wave operator.  By duality, the main point is to prove some hidden regularity 
\begin{align*}
    \nor{\frac{d v}{d\nu}}{L^2([0,t],L^2(\d \Omega))}\leq C \nor{(v(0),\partial_t v(0))}{H_0^1\times L^2(\Omega)},
\end{align*}
for some solution of the dual equation $\partial_t^2 v-\Delta v-a(x)\partial_t v=0$ with Dirichlet boundary condition. The proof is the same as in the free case \cite[Theorem 7.1.3.]{TucsnakWeiss}.
\enp
\begin{remark}
    It could be surprising at first that there is no link between $T$ and $T_{\omega}$, the time of geometric condition. That means that it does not directly come as a consequence of observability estimates that require a minimal time.
\end{remark}
Note that we could also think about the reverse process.  If we are targeting a specific function to return periodically, can we find a source term that provides the solution? We refer to  Pomet-Coron-Praly~\cite{CPP:90} for questions relating control and periodic solutions. In this context, it turns out to be a reformulation of a control problem and comes as a direct consequence of the result of Bardos-Lebeau-Rauch~\cite{BLR}.
\begin{theorem}
Under the same assumptions as Theorem~\ref{t:periodboundary}.  Let $\gamma \subset \d\Omega$ satisfying the Geometric Control Condition~\eqref{assumGCC} for a time $T_{\gamma}$. Let $T>T_{\gamma}$. Then, for any $(u_0,u_1)\in L^2\times H^{-1}(\Omega)$, there exists $f\in L^2_{per}([0,T],L^2(\d \Omega))$ supported in $\gamma$ so that the solution of \eqref{dampedwaveboundary} is $T$-periodic.
\end{theorem}
\bnp
We only need to find $\widetilde{f}\in L^2([0,T],L^2(\d \Omega))$ so that $(u,\partial_t u)(T)=(u_0,u_1)$ which is exactly the result of the control problem of \cite{BLR}. Defining $f\in L^2_{per}([0,T],L^2(\d \Omega))$ by periodization of $\widetilde{f}$. The uniqueness of the flow and the periodicity of $f$ provide the periodicity of $u$.
\enp
\subsection{Coupled parabolic/hyperbolic systems}
\label{s:exampleheatwave}
We follow the formulation of \cite{ZZ:07}. Let $\Omega\subset \R^n$ ($n \in\N$) be a bounded domain with $C^2$ boundary $\Gamma=\d\Omega$. Let $\Omega_H=\Omega_1$ be a subdomain of $\Omega$ and $\Omega_W=\Omega_2=\Omega\setminus \overline{\Omega_H}$ both connected. We denote by $\gamma$ the interface, $\Gamma_{j}=\Omega_j\setminus \overline{\gamma}$ and $\nu_j$ the unit outward normal vector to $\Omega_j$. We have $\Gamma=\overline{\Gamma_H\cup \Gamma_W}$. We assume that $\gamma$ is a non-empty open subset of $\d \Omega_j$ with positive capacity (see \cite{ZiemerBook} for the definition of capacity) and of class $C^1$ (unless otherwise stated).  We assume also $\operatorname{Cap}(\Gamma_W) > 0$ and $\operatorname{Cap}(\Gamma_H) > 0$.
\begin{assumption}
\label{a:obserwave}
   There exist two constants $T_0 > 0$ and $C > 0$ such that, for any $T> T_0$, all
solutions of the system
\bneq
\partial_t^2 y-\Delta_g y&=&0\quad \textnormal{ on }[0,T]\times \Omega\\
y&=&0 \quad \textnormal{ on }[0,T]\times \d \Omega\\
(y,\partial_t y)(0)&=&(y_0,y_1)\textnormal{ on } \Omega.
\eneq
satisfy 
\begin{equation*}
    |y_0|_{H^1_0(\Omega)}^2+|y_1|_{L^2(\Omega)}^2\leq C \int_0^T \int_{\Omega_H} |\partial_t y|^2 dx dt, \quad \forall (y_0,y_1)\in H^1_0(\Omega) \times L^2_0(\Omega).
\end{equation*}
\end{assumption}
When $\Omega$ is smooth enough, the assumption is satisfied if $\Omega_H\subset \Omega$ satisfies the Geometric Control Condition~\eqref{assumGCC} in $\Omega$ of Bardos-Lebeau-Rauch~\cite{BLR}. 

We are concerned with the following free system.
\bneqn
\label{waveheat}
\partial_t u-\Delta_g u=&0\quad &\textnormal{ on }[0,T]\times \Omega_H\\
\partial_t^2 w-\Delta_g w=&0\quad &\textnormal{ on }[0,T]\times \Omega_W\\
\partial_t w=&u\quad &\textnormal{ on }[0,T]\times \gamma\\
\partial_{\nu_1} u=-&\partial_{\nu_2}  w \quad &\textnormal{ on }[0,T]\times \gamma\\
w=&0 \quad &\textnormal{ on }[0,T]\times \Gamma_{W}\\
u=&0 \quad &\textnormal{ on }[0,T]\times \Gamma_H\\
u(0)=&u_0\quad &\textnormal{ on } \Omega_H\\
(w,\partial_t w)(0)=&(w_0,w_1)&\textnormal{ on } \Omega_W.
\eneqn
We define the restriction spaces
\begin{equation*}
H^1_{\Gamma_H}(\Omega_H)=\left\{h|_{\Omega_H}\left|h\in H^1_0(\Omega)\right.\right\},\quad H^1_{\Gamma_W}(\Omega_W)=\left\{h|_{\Omega_W}\left|h\in H^1_0(\Omega)\right.\right\}.
\end{equation*}
and the phase space
\begin{equation*}
    H=L^2(\Omega_H)\times H^1_{\Gamma_H}(\Omega_W)\times L^2(\Omega_W),
\end{equation*}
together with the norm
\begin{equation*}
    \nor{(u,w_0,w_1)}{H}^2= \nor{u}{L^2(\Omega_H)}^2+\nor{w_0}{L^2(\Omega_W)}^2+\nor{\nabla w_0}{L^2(\Omega_W)}^2+\nor{w_1}{L^2(\Omega_W)}^2,
\end{equation*}
when $\Gamma_W$ is a non empty open set of the boundary, the following semi-norm is actually a norm 
\begin{equation*}
    \left|(u,w_0,w_1)\right|_H^2= \nor{u}{L^2(\Omega_H)}^2+\nor{\nabla w_0}{L^2(\Omega_W)}^2+\nor{w_1}{L^2(\Omega_W)}^2.
\end{equation*}
The infinitesimal generator of the equation will be $A: \D(A)\subset H \longrightarrow H$ by 
$$ A Y=(\Delta u,w_1,\Delta w_0),$$
where $Y=(u,w_0,w_1)$, and 
\begin{align*}
    \D(A)=\left\{(u,w_0,w_1)\in H\right.&\left|  \Delta u\in L^2(\Omega_H), \Delta w_0\in L^2(\Omega_W), w_1\in H^1(\Omega_W),\right.\\
   & u|_{\Gamma_H}=0, \quad w_1|_{\Gamma_W}=0,\\
 &\left.  u|_{\gamma}=w_1|_{\gamma}, \quad \partial_{\nu} u|_{\gamma}=\partial_{\nu} w_0|_{\gamma}\right\},
\end{align*}

We will assume $\operatorname{Cap}(\Gamma_W) > 0$ so that the operator $A$ generates a contractive $C^0$ semigroup in $H$ with $0\in\rho(A)$, see \cite[Theorem 1.]{ZZ:07}. Zhang-Zuazua proved the following decay 
\begin{theorem}\cite[Theorem 11.]{ZZ:07}
\label{t:XuZuaz}
Assume $\Omega_W$ satisfies Assumption~\ref{a:obserwave}. Then, there is a constant $C>0$ such that for any $(u_0, w_0,w_1) \in \D(A)$ the solution of \eqref{waveheat} satisfies
\begin{equation*}
  \nor{( u(t),w(t), \partial_t w(t))}{H}\leq \frac{C}{t^{1/6}}\nor{(u_0,w_0,w_1)}{\D(A)},\quad \forall t >0 .
\end{equation*}
\end{theorem}
For some different geometric conditions, Duyckaerts proved a better decay.
\begin{theorem}\cite[Theorem 2.]{D:07}
\label{t:Duyck}Assume that $\Omega_W$ and $\Omega_H$ are smooth domains and that the $C^{\infty}$ manifolds $\gamma$, $\Gamma_W$ and $\Gamma_H$, each of them assumed not empty, have empty intersections
so that
\begin{equation*}
    \d\Omega_W =\gamma\cup \Gamma_W, \quad     \d\Omega_H =\gamma\cup \Gamma_H.
\end{equation*}
Assume that $\gamma\subset \d\Omega_W$ satisfies the \eqref{assumGCC} inside of $\Omega_W$. Then for any $s<1$, there is a constant $C_s>0$ such that for any $(u_0, w_0,w_1) \in \D(A)$ the solution of \eqref{waveheat} satisfies
\begin{equation*}
  \nor{( u(t),w(t), \partial_t w(t))}{H}\leq \frac{C}{t^{s}}\nor{(u_0,w_0,w_1)}{\D(A)}, \forall t >0 .
\end{equation*}
\end{theorem}
Note that the assumption  $\Gamma_W$ not empty ensures that the square of the energy $E_b$ described in the result is a norm equivalent to $\nor{\cdot}{H}$. We will now use these theorems and our abstract result to obtain periodic solutions for the forced system
\bneqn
\label{waveheatsource}
\partial_t u-\Delta_g u=&f\quad &\textnormal{ on }[0,T]\times \Omega_H\\
\partial_t^2 w-\Delta_g w=&g\quad &\textnormal{ on }[0,T]\times \Omega_W\\
\partial_t w=&u \quad &\textnormal{ on }[0,T]\times \gamma\\
\partial_{\nu} u=&\partial_{\nu}  w \quad &\textnormal{ on }[0,T]\times \gamma\\
w=&0 \quad &\textnormal{ on }[0,T]\times \Gamma_{W}\\
u=&0 \quad &\textnormal{ on }[0,T]\times \Gamma_H\\
u(0)=&u_0\quad &\textnormal{ on } \Omega_H\\
(w,\partial_t w)(0)=&(w_0,w_1)&\textnormal{ on } \Omega_W.
\eneqn
\begin{theorem}\label{t:perheatwave}
Under the assumptions of Theorem~\ref{t:XuZuaz}, then, for any $T>0$, $f\in  W^{7,1}_{per,0}([0,T],L^2(\Omega_H))$ and $g\in W^{7,1}_{per,0}([0,T],L^2(\Omega_W))$, there exists $(u_0,w_0,w_1)\in H$ so that the unique solution $(u,w,\partial_t w)\in C^0([0,T], H)$ to \eqref{waveheatsource} is $T$-periodic. 

Under the assumptions of Theorem~\ref{t:Duyck}, the same result holds for $f\in  W^{2,1}_{per,0}([0,T],L^2(\Omega_H))$ and $g\in W^{2,1}_{per,0}([0,T],L^2(\Omega_W))$.
\end{theorem}
\bnp
We want to apply Corollary~\ref{cor:decaypol} for $\X=H$ with the previously defined $A$ and its associated semigroup.  System~\eqref{waveheatsource} can be written $\partial_t X(t)=AX(t)+F(t)$ with $X(t)=( u(t),w(t), \partial_t w(t))\in \X$ and $F=(f,0,g)$. Theorem~\ref{t:XuZuaz} gives $\beta_0\geq 1/6$, so Corollary~\ref{cor:decaypol} works for $k>6$ and gives the existence of a periodic solution if the source term satisfies $F\in W^{7,1}_{per,0}([0,T],\X)$.  That is $f\in  W^{7,1}_{per,0}([0,T],L^2(\Omega_H))$ and $g\in W^{7,1}_{per,0}([0,T],L^2(\Omega_W))$.

Under the assumptions of Theorem~\ref{t:Duyck}, we obtain $\beta_0\geq 1$, which gives $k>1$ and a similar result.
\enp
Some results in the spirit of Theorem~\ref{t:perheatwave} have been obtained recently in Mosny-Muha-Schwarzacher-Webster~\cite{MMSW:24}.  The geometric assumptions differ but seem to be stronger and less natural than the Geometric Control Condition.  They are of multiplier type, which is often known to be more restrictive than the Geometric Control Condition \ref{assumGCC}, see Miller~\cite{M:02} in the context of control problems.  Concerning the regularity, they require, depending on some geometric condition, either
\begin{itemize}
    \item $f\in  H^3_{per}([0,T],H^1_{\Gamma_H}(\Omega_H)')$ and $g\in H^{6}_{per}([0,T],L^2(\Omega_W))$ to get some solution $(u,w,\partial_t w)\in H^{3}_{per}([0,T],H^1_{\Gamma_H}(\Omega_H))\times L^{\infty}_{per}([0,T],H^1_{\Gamma_W}(\Omega_H))\times L^{\infty}_{per}([0,T],L^2_{\Gamma_H}(\Omega_W))$.
    \item $f\in  H^4_{per}([0,T],H^1_{\Gamma_H}(\Omega_H)')$ and $g\in H^{8}_{per}([0,T],L^2(\Omega_W))$ to get some solution $(u,w,\partial_t w)\in H^{4}_{per}([0,T],H^1_{\Gamma_H}(\Omega_H))\times L^{\infty}_{per}([0,T],H^1_{\Gamma_W}(\Omega_H))\times L^{\infty}_{per}([0,T],L^2_{\Gamma_H}(\Omega_W))$.
\end{itemize}
We refer to the article for results at other regularity levels.  It seems that their regularity is more demanding than the second geometric configuration of Theorem~\ref{t:perheatwave} with smooth boundaries.  Our first case is slightly worse than their first case and slightly better than their second case for the wave part.  Regarding the heat part, the regularity they require is less demanding, suggesting that our result could be improved.  We believe it can be done by choosing a better admissible space, taking into account the regularizing effect of the heat equation, which allows a source term with less regularity in space but $L^2$ in time.
\section{A counterexample for slow decay}
The objective of this section is the construction of growing solutions announced in Theorem~\ref{t:counter}. In all this part, we are therefore in the configuration of the linear damped Schr\"odinger equation on $\mathbb{S}^2$ and the assumptions of Theorem~\ref{t:counter}, that is $a$ is zero in a neighborhood of one equator.
\bnp[Proof of Theorem~\ref{t:counter}] We fix $k\in \N$ for the rest of the proof. For any $f\in L^1_{per}([0,2\pi],H^k(\M))$, we will denote $u_f\in C([0,+\infty),L^2(\M))$ the solution of \eqref{dampedSchrodper}, that is $u_f(t)=\int_0^t e^{\A(t-s)}f(s)ds$ where $\A= -i  \Delta_g-a(x)$ is the generator of the damped Schr\"odinger equation.

For any $m\in \N^*$ and $n\in \N^*$, we denote 
\begin{align*}
 F_{m,n}=   \left\{f\in L^1_{per}([0,2\pi],H^k(\M))\textnormal{ s.t. } \nor{u_f( 2\pi n)}{L^2}\leq  m\nor{f}{L^1_{per}([0,T],H^k(\M))}\right\}.
\end{align*}
We observe that the application $f\mapsto u_f( 2\pi n)=\int_0^{2\pi n} e^{\A(2\pi n-s)}f(s)ds$ is linear continuous from $L^1_{per}([0,T],H^k(\M))$ to $L^2(\M)$. This implies that $F_{m,n}$ is a closed subset of $L^1_{per}([0,2\pi],H^k(\M))$. In particular, $F_{m}= \cap_{n\in \N^*}F_{m,n}$ is also a closed subset. We prove that $F_m$ has empty interior for any $m\in \N^*$.

Now, we fix $m\in \N^*$ and assume by contradiction that there exists $h\in L^1_{per}([0,T],H^k(\M))$ and $\e>0$ so that $B(h,\e)\subset F_m$, where $B(h,\e)$ is the open ball in $L^1_{per}([0,2\pi],H^k(\M))$. We apply the second part of Proposition~\ref{p:largegrowth} below with $M=4m \left(\nor{h}{L^1_{per}([0,2\pi],H^k(\M))}+ \e\right)/\e$ to obtain one function $f_M\in L^1_{per}([0,2\pi],H^k(\M))$ with $\nor{f_M}{L^1_{per}([0,2\pi],H^k(\M))}\leq 1$ and one $n\in \N^*$ so that the solution $u_{f_M}$ of \eqref{dampedSchrodper} satisfies $\nor{u_{f_M}( 2\pi n)}{L^2}\geq M$. Now, consider $z=h+\frac{\e}{2}f_M\in B(h,\e)\subset F_m$. $z\in F_m=\cap_{n\in \N^*}F_{m,n}$ implies $z\in F_{m,n}$ and therefore 
\begin{equation*}
\nor{u_z(2\pi n)}{L^2}\leq  m\nor{z}{L^1_{per}([0,2\pi],H^k(\M))}\leq m \left(\nor{h}{L^1_{per}([0,2\pi],H^k(\M))}+ \frac{\e}{2}\right).
\end{equation*}
By linearity, we have $u_z=u_h+\frac{\e}{2}u_{f_M}$. Using similarly that $h\in F_{m,n}$, which implies  $\nor{u_h(2\pi n)}{L^2}\leq m \nor{h}{L^1_{per}([0,2\pi],H^k(\M))}$, and $\nor{u_{f_M}(2\pi n)}{L^2}\geq M$, we get, the lower bound
\begin{align*}
  \nor{u_z(2\pi n)}{L^2}&\geq \frac{\e}{2}\nor{u_{f_M}(2\pi n)}{L^2}-\nor{u_h(2\pi n)}{L^2} \\
  &\geq \frac{\e}{2}M- m\nor{h}{L^1_{per}([0,2\pi],H^k(\M))} \\
  &\geq m \left(\nor{h}{L^1_{per}([0,2\pi],H^k(\M))}+ 2\e\right).
\end{align*}
This is a contradiction. 

So, we have proved that for any $m\in \N^*$, $F_m$ has empty interior.  By the Baire theorem, that proves that $F=\cup_{m\in \N^*}F_m$ has empty interior and also that $G:=F^c$ is a $G_{\delta}$ dense subset of $L^1_{per}([0,2\pi],H^k(\M))$. We prove that for any $f\in G$, $\left(\nor{u_f(2\pi n)}{L^2}\right)_{n\in \N^*}$ is unbounded where $u_f$ is the associated solution to \eqref{dampedSchrodper}. Fix  $f\in G$ and let $D>0$ arbitrary. First, we have that $0\in F_{m,n}$ for any $(m,n)\in (\N^*)^2$, so that $0\in F$. In particular, $f\in G$ implies $f\neq 0$. We select $m\in \N^*$ with $m\geq \frac{D}{\nor{f}{L^1_{per}([0,2\pi],H^k(\M))}}$. Since $f\in G=\cap_{m\in \N^*}F_m^c$, by definition, $f\notin F_m$. In particular, since $F_m=\cap_{n\in \N^*}F_{m,n}$, there exists $n\in \N^*$ so that $f\notin F_{m,n}$. That is $\nor{u_f(2\pi n)}{L^2}>  m\nor{f}{L^1_{per}([0,2\pi],H^k(\M))}\geq D$. For fixed $f\in G$, we have proved that for any $D>0$, there exists $n\in \N^*$ so that $\nor{u_f(2\pi n)}{L^2}> D$, which is the first expected result.

The absence of a periodic solution for such $f\in G$ of solution of \eqref{dampedSchrodperu0} for any $v_0\in L^2(\M)$ is a consequence of Proposition~\ref{p:convuniq} since the existence of a periodic solution would imply that the solution of \eqref{dampedSchrodper} would converge to this solution. But we can be more precise. Indeed, for the second part of the result, we make a similar decomposition as Proposition~\ref{p:convuniq}.  We decompose the solution $v$ of \eqref{dampedSchrodperu0} as $v=u_f+v_L$ where $u_f$ is the solution of \eqref{dampedSchrodper} and $v_L(t)=e^{t\mathcal{A}}v_0$, with $\mathcal{A}_S= -i  \Delta_g-a(x)$, is the solution of the damped equation \eqref{dampedSchrod}. In particular, since $(v_L(2\pi n))_{n\in \N}$ is bounded in $L^2(\M)$ but $(u_f(2\pi n))_{n\in \N}$ is unbounded, then $(v(2\pi n))_{n\in \N}$ is unbounded.
\enp
\begin{prop}
\label{p:largegrowth}
    Assume that we are in one of the geometric situations of Theorem~\ref{t:counter}, then, there exists a sequence of eigenvalues $(\lambda_j)_{j\in \N}$, $\lambda_j\to +\infty$, $\lambda_j\in \N$ and eigenfunctions $\Phi_{\lambda_j}$ of $-\Delta_g$ so that 
    \begin{equation}
    \label{e:propeigen}
    \begin{split}
              -  \Delta_g \Phi_{\lambda_j}&=\lambda_j\Phi_{\lambda_j},\\
         \nor{ \Phi_{\lambda_j}}{L^2(\M)}&=1,\\
      \nor{a(x) \Phi_{\lambda_j}}{L^2(\M)}&\leq e^{-c\sqrt{\lambda_j}},
    \end{split}     
    \end{equation}
for some $c>0$. Moreover, for any $k\in \N$ and $M>0$, there exists $ f\in L^1_{per}([0,2\pi],H^k(\M))$ with $\nor{f}{L^1_{per}([0,2\pi],H^k(\M))}\leq 1$ and $n\in \N$, so that the solution $u_f$ of \eqref{dampedSchrodper} satisfies $\nor{u_f(2\pi n)}{L^2}\geq M$.
\end{prop}
\bnp
The construction of the concentrating eigenfunctions $\Phi_{\lambda_j}$ is actually very explicit and classical.  We refer, for instance, to \cite[Section 3.A.]{LL:21heat}. We consider $\mathbb{S}^2$ embedded in $\R^3$:
$$
\mathbb{S}^2 =  \{(x_1,x_2,x_3) \in \R^3 , x_1^2 + x_2^2 + x_3^2 = 1\} =  \{ x \in \R^3 ,|x|= 1\} .
$$
 Eigenfunctions and eigenvalues of the Laplace-Beltrami operator on $\mathbb{S}^2$ are well-understood: eigenfunctions are restrictions to $\mathbb{S}^2$ of harmonic homogeneous polynomials of $\R^3$, associated to the eigenvalue $j(j+1)$, where $j$ is the degree of the polynomial.  

We will use the so called equatorial spherical harmonics, given by
$$
u_j = P_j |_{\mathbb{S}^2}  \in C^\infty(\mathbb{S}^2) , \quad P_j(x_1,x_2,x_3) = (x_1 + i x_2)^j , \quad j\in \N,
$$
known to concentrate exponentially on the equator given by $x_3 = 0$.  They satisfy 
\begin{equation*}
-\Delta_{\mathbb{S}^2} u_j=\lambda_j u_j \textnormal{ with } \lambda_j=j(j+1)
\end{equation*}
and $\nor{u_j}{L^2(\mathbb{S}^2)}\sim 2^{1/2}\pi^{3/4}j^{-1/4}$ as $j \to + \infty$. Denoting $N=(0,0,1)$ the North pole and using the geodesic distance to define the ball $B(N,r)$, we have the following estimate (see 
\cite[Lemma 3.2.]{LL:21heat}) for all $k \in \N^*$ and all $r \in [0,\frac{\pi}{2})$,  
$$
\nor{u_j}{L^2(B(N,r))}^2= \frac{\pi}{j+1}\frac{\sin(r)^{2j+2}}{\cos(r)}\left(1+R \right) , \quad \text{ with }\quad |R| \leq  \frac{\tan(r)^2}{2j+2} .
$$
The same holds for the south pole $B(S,r)$. Then, we choose $\Phi_{\lambda_j}=c_i u_j$, with the normalization constant $c_j=\nor{u_j}{L^2(\mathbb{S}^2)}^{-1}\sim 2^{-1/2}\pi^{-3/4}j^{1/4}$ as $j \to + \infty$. Since we assumed that $a$ is zero in a neighborhood of the equator $x_3 = 0$ (up to rotation), they is some $r \in [0,\frac{\pi}{2})$ so that $\operatorname{supp}(a)\subset B(N,r)\cup B(S,r)$. We obtain for $j$ large enough
\begin{align*}
    \nor{a(x) \Phi_{\lambda_j}}{L^2(\M)}&\leq c_j \nor{a}{L^{\infty}(\M)} \nor{u_j}{L^2(B(N,r)\cup B(S,r))}\lesssim \frac{j^{1/4}}{\sqrt{j+1}} \sin(r)^{j+1}\lesssim e^{-d(j+1) }\lesssim e^{-d\sqrt{\lambda_j}},
\end{align*}
with $d=-\log(\sin(r))>0$ since $r \in [0,\frac{\pi}{2})$. Up to changing the subsequence so that $j$ is large enough and the previous estimates hold, we obtain the first part of the expected results, up to taking $c=d/2$. 

Now, we get back to the existence of a growing solution and write $T=2\pi$.  We denote $\A= -i  \Delta_g-a(x)$ the generator of the damped equation and $\A_0= -i\Delta_g$ the generator of the free Schr\"odinger equation.  We see that $\Phi_{\lambda_j}$ is an eigenfunction of $\A_0$ and $e^{\A_0 t}\Phi_{\lambda_j}=e^{i\lambda_j t }\Phi_{\lambda_j}$. In particular, since $T=2\pi$ and $\lambda_j\in \N$, $e^{ m T\A_0}\Phi_{\lambda_j}=\Phi_{\lambda_j}$ for any $m\in \N$.

 We pick $f(s)=f_j(s)=C_j e^{(s-T)\A} \Phi_{\lambda_j}/T$ for $s\in [0,T)$ with $j$ and $C_j>0$ to be chosen later on and we extend $f$ by periodicity. 
With the selected choice, we have $F_{T}=\int_{0}^{T}e^{\A(T-s)}f(s)ds=C_j\Phi_{\lambda_j}$. Using \eqref{e:unT}, the solution $u=u_j$ of \eqref{dampedSchrodper} satisfies
    \begin{equation}
    \label{e:decompsol}
    \begin{split}
    u(nT)&=\sum_{m=0}^{n-1}e^{ m T\A}F_T=C_ij\sum_{m=0}^{n-1}e^{ m T \A}\Phi_{\lambda_j}\\
  &=C_j\sum_{m=0}^{n-1}e^{ m T\A_0}\Phi_{\lambda_j} +C_j\sum_{m=0}^{n-1}\left(e^{mT\A }-e^{mT\A_0}\right)\Phi_{\lambda_j}\\
    &=C_j n\Phi_{\lambda_j} +C_j\sum_{m=0}^{n-1}\left(e^{mT\A }-e^{m T\A_0}\right)\Phi_{\lambda_j}.  
    \end{split}     
    \end{equation}
We now estimate the second term.  Let $v$ be the solution of 
\bneq
i\partial_t v-\Delta_g v+ia(x)v &=&0\\
v(0)&=&\Phi_{\lambda_j}
\eneq
and $w=e^{i\lambda_j t}\Phi_{\lambda_j}$ solution of
\bneq
i\partial_t w-\Delta_g w&=&0\\
w(0)&=&\Phi_{\lambda_j}.
\eneq
That is $ v(t)=e^{\A t}\Phi_{\lambda_j}$ and $ w(t)=e^{t\A_0}\Phi_{\lambda_j}=e^{i\lambda_j t}\Phi_{\lambda_j}$. $r=v-w$ is the solution of
\bneq
i\partial_t r-\Delta_g r+ia(x)r&=&-ia(x)w\\
r(0)&=&0.
\eneq
In particular, $ r(t)=\int_0^t e^{\A(t-s)}M(s)ds$ with $M(s)= -ia(x)w(s)=-ie^{i\lambda_j s}a(x)\Phi_j$. Since $a(x)\geq 0$, the $L^2$ norm of solutions of \eqref{dampedSchrod}, which gives $\nor{e^{\A(t-s)}}{\mathcal{L}(L^2)}\leq 1$, uniformly for $t-s\geq 0$. This allows to obtain the following bound, uniform in $t\geq 0$,
\begin{align*}
    \nor{r(t)}{ L^2}&\leq \int_0^t \nor{e^{\A(t-s)}M(s)}{ L^2}ds\leq \int_0^t \nor{M(s)}{ L^2}ds \leq \int_0^t \nor{a(x)\Phi_j}{L^2}ds\leq t  e^{-c\sqrt{\lambda_j}},
\end{align*}
where we have used \eqref{e:propeigen}. Notice that $\left(e^{mT\A }-e^{mT\A_0}\right)\Phi_{\lambda_j}= r(mT)$, so that 
\begin{align*}
    \nor{\left(e^{mT\A }-e^{mT\A_0}\right)\Phi_{\lambda_j}}{ L^2}& \leq mT  e^{-c\sqrt{\lambda_j}}.
\end{align*}
Getting back to \eqref{e:decompsol}, we obtain 
\begin{equation}
\label{e:lowerbound}
    \begin{split}
  \nor{  u(nT)}{L^2} &\geq C_j n\nor{\Phi_{\lambda_j}}{ L^2} -C_j\sum_{m=0}^{n-1}mT  e^{-c\sqrt{\lambda_j}}\\
  &\geq C_j n  \left(1 -nTe^{-c\sqrt{\lambda_j}}\right), 
    \end{split}
    \end{equation}
    where we have used $\nor{\Phi_{\lambda_j}}{ L^2}=1$. Now, we are ready to make some suitable choices to make the lower bound \eqref{e:lowerbound} as large as expected, keeping $f$ bounded. Since $a$ is smooth, the multiplication by $a(x)$ applies $H^k(\M)$ into itself and standard energy estimates allow to obtain a constant $C>0$ so that we have the bound $\nor{e^{(s-T)\A} v}{H^k(\M))}\leq C\nor{v}{H^k(\M))} $ for all $s\in [0,T]$, $v\in H^k(\M)$. So, we have
    \begin{align*}
    \nor{f}{L^1_{per}([0,2\pi],H^k(\M))}\leq  \frac{C_j}{T} \int_0^T \nor{e^{(s-T)\A} \Phi_{\lambda_j}}{H^k(\M)}\leq C \frac{C_j}{T} \int_0^T \nor{\Phi_{\lambda_j}}{H^k(\M)}\leq C  C_j \lambda_j^{k/2}.
    \end{align*}
    So, we select $C_j= C^{-1}  \lambda_j^{-k/2}$ in order to have $\nor{f}{L^1_{per}([0,2\pi],H^k(\M))}\leq 1$. Now, we select $n:=n_j=\lceil \lambda_j^{k/2+1}\rceil $. For $j$ large enough, $n_jTe^{-c\sqrt{\lambda_j}}=2\pi \lceil \lambda_j^{k/2+1}\rceil e^{-c\sqrt{\lambda_j}}\leq 1/2$. The lower bound \eqref{e:lowerbound} becomes
\begin{equation*}
 \nor{  u(n_jT)}{L^2} \geq C^{-1}  \lambda_j^{-k/2} \lceil \lambda_j^{k/2+1}\rceil  \left(1 -n_jTe^{-c\sqrt{\lambda_j}}\right)\geq C^{-1}  \lambda_j/2 . 
    \end{equation*}
 Taking $j$ large enough gives the expected result.   
\enp
\appendix
\section{Notations and Definitions}
\label{s:not}

\begin{definition}\label{uniformlyS}
    The system $\partial_t u=Au$, with $A$ an operator that generates a semigroup $\SS$ that follows Assumptions~\ref{assumptions}, is said 
    \emph{uniformly stable} if  
    $$\nor{e^{At}}{\mathcal{L}(\X)} \underset{t\to +\infty}{\longrightarrow}0.$$ 
\end{definition}
    Notice that this  will imply that  there exists $C>0$ such that 
     $\nor{e^{At}}{\mathcal{L}(\X)}\leq Ce^{-\lambda t}$.

\begin{definition}\label{stronglyS}
    The system $\partial_t u=Au$, with $A$ an operator that generates a semigroup $\SS$ that follows Assumptions~\ref{assumptions}, is said to be \emph{strongly stable}, if  $$\nor{e^{At}u_0}{\X} \underset{t\to +\infty}{\longrightarrow}0$$ for any $u_0\in \X$.
\end{definition}
   For a linear operator as assumed in the introduction, $\D(A)$ denotes the domain of $A$. Since $A$ is closed, this is a Banach space when given the norm 
\begin{equation*}
    \nor{x}{\D(A)}=\nor{x}{\X}+\nor{Ax}{\X}.
\end{equation*}
 $\rho(A)$ denotes the resolvent set, that is the set of $\lambda\in \C$ so that $\lambda \id-A$ is invertible from $\D(A)$ to $\X$ with inverse $R(\lambda,
A)=(\lambda \id-A)^{-1}$. $\sigma(A)$ is the spectrum, that is $\C\setminus \rho(A)$.

We assume $0\in \rho(A)$, which allows to define $A^{-1}$ from $\X$ to $\D(A)$.
$\X_{-1}$ is the completion of $\X$ for the norm $\nor{x}{\X_{-1}}=\nor{A^{-1}}{\X}$.\\

$\lceil \alpha\rceil\in \N$ is the superior integer value of $\alpha>0$, that is satisfying $\lceil \alpha\rceil-1< \alpha \leq  \lceil \alpha\rceil$.\\ 

\begin{definition}{\bf Useful spaces of periodic functions}\\    
For some Banach space $E$, we denote $L^p_{per}([0,T],E)$ the space of functions $f\in L^p_{loc}(\R,E)$ that satisfy $f(t+T)=f(t)$, almost everywhere. It is a Banach space with the norm $\nor{f}{L^p_{per}([0,T],E)}=\nor{f|_{[0,T]}}{L^p([0,T],E)}$. \\

We make similar definitions for 
\begin{itemize}
    \item $W^{k,p}_{per}([0,T],E)$ for $k\in \N$, $p\in [0,\infty]$,
  \item   $H^{k}_{per}([0,T],E)=W^{k,2}_{per}([0,T],E)$,
    \item $C^{k}_{per}([0,T],E)$.
\end{itemize}
For $k\in \N^*$, $p\in [0,\infty]$, we also define the space $W^{k,p}_{per, 0}([0,T],E)$ as follows
\begin{align}
\label{e:defper0}
 W^{k,p}_{per, 0}([0,T],E)=  \left\{ f\in W^{k,p}_{per}([0,T],E); f^{(j)}(0)=f^{(j)}(T)=0, \quad 0\leq j\leq k-1\right\}.
\end{align}
\end{definition}
In particular, for functions in $ W^{k,p}_{per, 0}([0,T],E)$, the integration by parts in time, when it is licit, does not produce boundary terms; see, for instance, Lemma~\ref{lm:gainder}. Note that, $L^p_{per}([0,T],E)$ can be identified with $L^p([0,T],E)$ by the restriction, but the space $W^{k,p}_{per}([0,T],E)$ can not be identified with $W^{k,p}([0,T],E)$ because the trace at $0$ and $T$ should be the same.

Also, these functions can be equivalently considered on the  torus $\mathbb{T}_T=\R/T\Z$, but it will sometimes be convenient to consider them as functions of $\R$, for instance when we will consider Duhamel formula.

\section{Resolvent and decay}
\begin{lemma}\cite[Lemma 2.1.]{BT:10}
    Let $\SS$ be a bounded $C^0$-semigroup on a Banach space
$\X$ with generator $A$. Then $\C_+\subset \rho(A)$ with $\C_+=\{z\in \C; \operatorname{Re}(z)>0\}$.
\end{lemma}
The following result is written this way in Borichev-Tomilov~\cite[Theorem 1]{BT:10}, but seems to originate in Arendt-Batty~\cite{AB:88} and Batty~\cite{B:90}.
\begin{theorem}\label{t:decay}
    Let $(S(t))_{t\geq 0}$ be a bounded $C^0$-semigroup on a Banach space
$\X$ with generator $A$. Suppose that $i\R$ is contained in the resolvent set $\rho(A)$
of $A$. Then
\begin{align*}
    \nor{S(t)A^{-1}}{\mathcal{L}(\X)}\underset{t\to+\infty}{\longrightarrow}0.
\end{align*}
\end{theorem}
For $A$ as before,  we define the continuous non-decreasing
function
\begin{equation}\label{mf}
M(\eta)=\max_{t \in [-\eta,\eta]}\nor{R(it, A)}{\mathcal{L}(\X)},\qquad \eta \ge 0,
\end{equation}
and the associated modified function
\begin{equation}\label{mflog}
M_{\log}(\eta):=M(\eta)\big(\log(1+M(\eta))+\log (1+\eta)\big),\quad \eta
\ge 0.
\end{equation}
Let $M^{-1}_{\log}$ be the inverse of $M_{\log}$ defined on
$[M_{\log}(0), +\infty).$
\begin{theorem}[Batty-Duyckaerts~\cite{BD:08}]\label{t:BD}
Let $\SS$ be a bounded $C_0$-semi\-group on a Banach
space $\X$ with generator $A,$ such that $i \mathbb R \subset
\varrho (A).$ Let the functions $M$ and $M_{\log}$ be defined by
\eqref{mf}
 and \eqref{mflog}. Then
there exist $C,B>0$ such that
\begin{equation*}
 \nor{S(t) A^{-1}}{\mathcal{L}(\X)}
\le \frac{C}{M^{-1}_{\log}(t/C)}, \quad t \ge B.
\end{equation*}
\end{theorem}

Note that in the case $\alpha>0$, $M(\eta) \le C (1+\eta^{\alpha})$, $\eta\ge 0$,
the Batty-Duyckaerts result gives \begin{equation*}
 \nor{S(t) A^{-1}}{\mathcal{L}(\X)} \le C
\Bigl( \frac{\log t} {t}  \Bigr)^{\frac{1}{\alpha}}, \quad t \ge B.
\end{equation*}
\begin{theorem}[\cite{BT:10}]
Let  $\SS$ be a bounded $C_0$-semigroup on a Hilbert
space $\X$ with generator $A$ such that $i \mathbb R \subset
\varrho (A).$ Then for a fixed $\alpha > 0$ the following
conditions are equivalent:
\begin{enumerate}
\item 
\begin{eqnarray*}
 \nor{R( is , A)}{\mathcal{L}(\X)}= {\rm O} (|s|^{\alpha}), \qquad s \to \infty.
\end{eqnarray*}
\item 
\begin{eqnarray*}
 \nor{S(t)A ^{-1}}{\mathcal{L}(\X)}= {\rm O} (t^{-1/\alpha}), \qquad t \to \infty.
\end{eqnarray*}
\end{enumerate}
\end{theorem}
\section{Fractional powers of operators and interpolation Spaces}
\label{s:powerinterp}
In this section, we define and state some classical results concerning the fractional powers of operators.  The results are taken from Komatsu~\cite{K:66,K:67}.  It will require for the resolvent that $(-\infty,0)\subset\rho(\B) $, as well as, for $\lambda>0$ in the resolvent of $\B$,  $(\lambda \id+\B)^{-1}$ and $\lambda(\lambda \id+\B)^{-1}$ are uniformly bounded operator.\\This will be applied to the operator $B=-A$, which satisfies these assumptions to define its power under Assumption~\ref{assumptions}, see \cite[Section 11]{K:66}.

We treat the operator $B^{\alpha}$ and its domain $\D(\B^\alpha)$ when $\alpha>0.$  For $B$ a closed linear operator defined in the Banach space $\X$, $\overline{\D(\B)}=\X$.

To understand the interpolation required in Theorem \ref{t:interp}, let us consider  the intermediate spaces $ \D_p^{\sigma}$, for $\sigma\in (0,\infty)$ and $1\le p \le \infty$. The Banach space  $ \D_p^{\sigma}$ defined by 
$$\D_p^{\sigma}=\{ x\in \X : \lambda^{\sigma}(\B(\lambda \id+\B)^{-1})^m x \in L^p(\X) \}$$
for $m\in \N$ such that $m>\lfloor \sigma\rfloor$,  with the norm $\|x\|_{\X}+ \|\lambda^{\sigma}\B(\lambda \id+\B)^{-1})^m x\|_{L^p(\X)}$. Here, $L^p(\X)$ is the space of $\X$ valued functions with respect to the measure $\frac{d\lambda}{\lambda}$ on $(0,+\infty)$. And the operators  $\B_{\sigma}^{\alpha}:D_1^{\sigma} \mapsto\X$ with $\sigma \geq \alpha>0$ is defined for $x\in D_1^{\sigma}$ as
    $$\B_{\sigma}^{\alpha}x=\frac{\Gamma(m)}{\Gamma(\alpha)\Gamma(m-\alpha)}\int_0^{\infty} \lambda^{\alpha -1}(\B(\lambda \id+\B)^{-1})^m x d\lambda.  $$
$D_1^{\sigma}$ and $\B_{\sigma}^{\alpha}$ as above are shown to be independent on $m>\sigma$.    
   \begin{definition}\cite[Definition 2.1.]{K:67}
For $\alpha>0$, the fractional power $\B^{\alpha}$ is the smallest closed extension of $\B_{\sigma}^{\alpha}$ for $\sigma\geq \alpha$.
    \end{definition}
This is known not to depend on $\sigma$ and to coincide with $\B^m$ when $m\in \N^*$.  It also coincides with the one defined in \cite[Equation 1.2]{K:66} by taking the closed extension of the operator defined by  
\begin{gather*}
\B^\alpha_{+}:= -\frac{\sin \pi \alpha}{\pi} \int_0^\infty \lambda^\alpha(\lambda \id +\B)^{-1}d\lambda,\end{gather*}
 for $x$ in a suitable domain, see \cite{K:67} for details. The interpolation theorem presented by Komatsu in \cite{K:67} is
\begin{theorem} \cite[Theorem 3.1]{K:67}
    The space $S(p,\theta,\X; p,\theta-1,\D(\B^m))$ for $0<\theta<1$ and $1\le p \le \infty$ coincides with $\D_p^{\theta m}(\B)$,
    where the interpolation space $S(p,\theta,\X; p,\theta-1,\D(\B^m))$ is defined by Lions-Peetre in \cite{PL:64}.
\end{theorem}
\begin{lemma}\cite[Theorem 2.8]{K:67}\label{lminclusp}
    For every $\alpha>0$, 
    \begin{eqnarray*}
        \D_1^{\alpha}\subset \D(\B^{\alpha})\subset \D_{\infty}^{\alpha}.
    \end{eqnarray*}
\end{lemma} 
In particular, combining the two previous theorems, we obtain 
 \begin{eqnarray}
 \label{sandwichinterp}
  S(1,\theta,\X; 1,\theta-1,\D(\B^m))   \subset \D(\B^{\theta m})\subset S(\infty,\theta,\X; \infty,\theta-1,\D(\B^m)).
    \end{eqnarray}
Note that $\D(\B^{\theta m})$ can sometimes be equal to the complex interpolation space, but it seems to require more restrictive conditions, see for instance \cite[Theorem 2.30, Theorem 16.3]{YagiBook}.

This allows to obtain the following interpolation result.
\begin{lemma}\label{lm:inter}
 For any $0<\theta<1$ and $m\in \N^{*}$, there exists $D_{m,\theta}$ so that for any $T\in \mathcal{L}(\X)$ satisfying
    \begin{align*}
        \nor{T}{\mathcal{L}(\X)}\leq C_0; \quad \nor{T}{\mathcal{L}(\D(B^m),\X)}\leq C_m,
    \end{align*}
   we have
     \begin{align*}
    \nor{T}{\mathcal{L}(\D(B^{\theta m}),\X)}\leq D_{m,\theta} C_0^{1-\theta}C_m^{\theta}.
    \end{align*} 
\end{lemma}
\bnp
By the interpolation property (see \cite[Theorem 3.1]{PL:64}) applied with $p=\infty$ and  using  that $S(\infty,\theta,\X; \infty,\theta-1,\X)=\X$, we get that 
    \begin{align*}
    \nor{T}{\mathcal{L}( S(\infty,\theta,\X; \infty,\theta-1,\D(\B^m)),\X)}\leq  C_0^{1-\theta}C_m^{\theta}.
    \end{align*} 
    The embedding \eqref{sandwichinterp} allows to conclude with $D_{m,\theta}$ the constant of the embedding $\D(\B^{\theta m})\subset S(\infty,\theta,\X; \infty,\theta-1,\D(\B^m))$.
\enp
\begin{lemma}
\label{lm:iterpower}
  For any $m\in \N$, there exists $D_m$ so that for any $T\in \mathcal{L}(\X)$, commuting with $\B$, so that 
    \begin{align*}
 \nor{T}{\mathcal{L}(\D(\B),\X)}\leq C_1,
    \end{align*}
    we have,
     \begin{align*}
    \nor{T^m}{\mathcal{L}(\D(\B^{m}),\X)}\leq D_m  C_1^{m}.
    \end{align*} 
\end{lemma}
\bnp
The result is obtained by iteration. $m=0$ is direct with $D_0=1$. Assume the result is true at step $m$.  For $x\in \D(\B^{m+1})$, we write by iteration step and using the assumption on $T$
\begin{align*}
    \nor{T^{m+1}x}{\X}=&\nor{T^{m}Tx}{\X}\leq D_m  C_1^{m}\nor{Tx}{\D(\B^{m})}\leq   D_m C_1^{m}\left(\nor{Tx}{\X}+\nor{\B^{m}Tx}{\X}\right)\\
    \leq  &  D_m  C_1^{m}\left(\nor{Tx}{\X}+\nor{T \B^{m}x}{\X}\right)\\
    \leq &  D_m  C_1^{m+1}\left(\nor{x}{\D(\B)}+\nor{ \B^{m}x}{\D(\B)}\right)\\
    \leq &  D_m  C_1^{m+1}\left(\nor{x}{\X}+\nor{\B x}{\X}+\nor{\B^{m}x}{\X}+\nor{\B^{m+1}x}{\X}\right)\\
    \leq & D_{m+1 } C_1^{m+1}\nor{x}{\D(\B^{m+1})},
\end{align*}
where $D_{m+1}$ is a product of $D_m$ and some constants given by Lemma~\ref{equivnorm} below.
\enp
\begin{lemma}\label{equivnorm}
For any $m\in \N^*$ and $0\leq k\leq m$ there exists $C>0$ so that 
\begin{equation*}
    \nor{\B^k x}{\X}\leq C \nor{x}{\D(\B^{m})}, \quad \forall x\in \D(\B^{m}).
\end{equation*}
Moreover, if $\alpha>\beta>0$, then $\D(\B^{\alpha})\subset \D(\B^{\beta})$.
\end{lemma}
This is a direct consequence of \cite[Theorem 2.5]{K:66}, see the comments on page 294, stating the equivalence of the norms. 
\section*{Acknowledgement}
The authors would like to thank one more time Jean-Michel Coron for his kindness and everything he has done for the community of control over the  years. 

 They would also like to thank Matthieu L\'eautaud for bringing \cite{L:23} to their attention, which led to a cleaner presentation of their result for the damped wave equation in the case  $\d\M=\emptyset$. Finally, the authors are very grateful to the anonymous referees, whose careful reading substantially improved the presentation of the paper.
\bibliographystyle{plain}

\bibliography{biblio2}

\end{document}